\documentclass[12pt]{article}

\usepackage{amsmath,amsthm,amssymb,latexsym,amscd,mathtools}
\usepackage{setspace}

\usepackage{graphicx}

\usepackage{stmaryrd}
\usepackage{color}
\usepackage{pgf,tikz}
\usepackage{multirow}
\usepackage{bm}

\usetikzlibrary{arrows,patterns}
\usepackage{rotating}
\usepackage{mathtools}
\usepackage{enumitem}
\usepackage{comment}

\newcommand{\white}[1]{{\textcolor{white}{#1}}}
\newcommand\scalemath[2]{\scalebox{#1}{\mbox{\ensuremath{\displaystyle #2}}}}
\newcommand{\R}{{\mathbb R}}

\newcommand{\N}{{\mathbb N}}
\newcommand{\Z}{{\mathbb Z}}

\newcommand{\supp}{\text{supp}}

\newcommand{\IL}{I_{\Lambda}}

\newcommand{\conv}{\text{conv}}
\newcommand{\hull}{\text{hull}}
\newcommand{\im}{\text{Im}}

\newcommand{\otimesS}{\otimes_{S[\Lambda]}}

\theoremstyle{plain}
\newtheorem{theorem}{Theorem}[section]
\newtheorem{example}{Example}[section]
\newtheorem{lemma}{Lemma}[section]
\newtheorem{definition}{Definition}[section]
\newtheorem{prop}{Proposition}[section]
\newtheorem{corollary}{Corollary}[section]

\newtheorem{remark}{Remark}[section]

\definecolor{ffffff}{rgb}{1,1,1}
\definecolor{wwwwww}{rgb}{0.4,0.4,0.4} 

\definecolor{wwwwqq}{rgb}{0.4,0.4,0}
\definecolor{zzqqtt}{rgb}{0.6,0,0.2}
\definecolor{qqqqff}{rgb}{0,0,1}
\definecolor{uququq}{rgb}{0.25,0.25,0.25}
\definecolor{cccccc}{rgb}{0.8,0.8,0.8}
\definecolor{zzzzzz}{rgb}{0.6,0.6,0.6}
\definecolor{tttttt}{rgb}{0.2,0.2,0.2}
\definecolor{zzttqq}{rgb}{0.27,0.27,0.27}
\definecolor{qqwwcc}{rgb}{0.4,0.4,0.4}
\definecolor{fffftt}{rgb}{0.73,0.73,0.73}
\definecolor{qqzzqq}{rgb}{0.2,0.2,0.2}
\definecolor{ffqqqq}{rgb}{0.33,0.33,0.33}

\begin{document}

\title{A Combinatorial Algorithm to Find the Minimal Free Resolution of an Ideal with Binomial and Monomial Generators}
\date{September 29, 2014}
\author{Trevor McGuire\thanks{Research done under James Madden at Louisiana State University with partial support from GAANN (NSF award P200A120001) and VIGRE (NSF award 0739382).}\\ North Dakota State University}

\maketitle

\begin{abstract}

 %\addcontentsline{toc}{chapter}{Abstract\dotfill}

%\doublesize

In recent years, the combinatorial properties of monomials ideals \cite{H1,MS,S1} and binomial ideals \cite{MK,MS} have been widely studied.  In particular, combinatorial interpretations of free resolution algorithms have been given in both cases.  In this present work, we will introduce similar techniques, or modify existing ones to obtain two new results.  The first is $S[\Lambda]$-resolutions of $\Lambda$-invariant submodules of $k[\Z^n]$ where $\Lambda$ is a lattice in $\Z^n$ satisfying some trivial conditions.  A consequence will be the ability to resolve submodules of $k[\Z^n/\Lambda]$, and in particular ideals $J$ of $S/\IL$, where $\IL$ is the lattice ideal of $\Lambda$.

Second, we will provide a detailed account in three dimensions on how to lift the aforementioned resolutions to resolutions in $k[x,y,z]$ of ideals with monomial and binomial generators.
\end{abstract}

\section{Introduction}
In recent decades, various groups of mathematicians have independently studied resolutions of binomial ideals, and resolutions of monomial ideals.  Many beautiful results have been obtained, but resolutions of sums of such ideals remain elusive.  It is exactly these types of ideals that will be studied in this present work.

In the first section, we will discuss the combinatorial setup we will be using for the rest of work.  The objects of interest  are subsets of $\Z^n$ that are typically infinite. (In the existing theory, researchers utilized finite subsets of $\N^n$.)  We will draw on the language of \cite{BHS} to generalize the tools from \cite{ES} and \cite{MS}.

The next section examines subsets of $\Z^n$ that are groups as well as antichains.  We will call them antichain lattices, and we will work intimately with them throughout the remainder of the work.  Our antichain condition parallels other work where the subgroups are not allowed to intersect the positive orthant anywhere but 0; requiring that the lattice is an antichain is a more concise way to state this condition.

We will give a brief review of resolutions in the following sections, specifically focusing on resolutions of certain types of binomial ideals that have been studied in \cite{ES} and \cite{MS}.  

The penultimate section will take us on our final step before we begin resolving our desired ideals.  We will need to enter the world of Laurent monomial modules, which is the analogue of monomial ideals, but in a larger ambient space.  We will look at $k[x_1,\dots,x_n]$-modules contained in the Laurent polynomial ring over $k$.

The final section, where the bulk of the new work lies, will tie everything together in the full generality of $\Z^n$, but our final computation will actually be in $\Z^3$ because the the increasingly complex computations in $\Z^n$ do not lend themselves to concise notation.  That is, we will give the general combinatorial algorithm for the resolution of certain ideals with binomial and monomial generators in $k[x_1,x_2,x_3]$ as the main result. We will conclude with a detailed example outlining the full algorithm.

\section{Subsets of $\Z^n$}

The general setup we will be working with is one of $M$-sets, where $M$ is a monoid.  

\begin{definition}
 Let $M=<M,\ast, 0>$ be a monoid.  Then an $M$-set is a set $S$ together with a map $$\begin{array}{ccc}M\times S & \rightarrow & S \\ (m,s) & \mapsto & ms\end{array}$$ such that $(m\ast m')s=m(m's)$ and $0s=s$.
\end{definition}
\vspace{.3cm}

\subsection{Subsets of $\Z^n$ as a Poset}
We have the following definitions and notations for elements $\alpha, \beta$ and subsets $A$ of $\Z^n$:
\singlespace
\begin{enumerate}
\item If $\alpha\in\R^n$, then $\pi_j(\alpha)$ denotes the $j^{th}$ component of $\alpha$.
\item $\alpha \leq \beta$ if $\pi_i(\alpha) \leq \pi_i(\beta)$, $i=1, \dots, n$
\item $\alpha < \beta$ if $\alpha \leq \beta$, and $\alpha \neq \beta$
\item $\alpha << \beta$ if $\pi_i(\alpha) < \pi_i(\beta)$, $i=1, \dots, n$\footnote{At times, we use the notation $a<< b$ for $a,b\in\R$ to mean that $b$ is much greater than $a$, but context will prevent any notational confusion.}
\item $\min(A) := \{\alpha \in A | \zeta < \alpha \Rightarrow \zeta \notin A\}$
\item If $A = A + \N^n$, then $A$ is an $\N^n$-set with the map being defined by $(\eta,\alpha)\mapsto \eta\alpha=\alpha+\eta$.
\item The $\N^n$ set generated by $A$ is $A+\N^n=\{\zeta \in \Z^n | \exists \alpha \in A \text{ with } \alpha \leq \zeta\}$
\item If $\alpha,\beta\in\Z^n$, then $\alpha\vee\beta=(\text{sup}\{\alpha_1,\beta_1\}, \dots, \text{sup}\{\alpha_n,\beta_n\})$, and $\alpha\wedge\beta=-(-\alpha\vee-\beta)$.

\end{enumerate}
\doublespace

\begin{definition}
A descending chain in a poset $X$ is a function $f:I\rightarrow X$ where $I\subseteq \N$ is an interval and $f(i)>f(j)$ if $i<j$.  If $A\subseteq X$ does not have any infinite descending chains, we will say it satisfies the decending chain condition, and we call it a DCC set.
\end{definition}

If $A\subseteq\Z^n$ is a DCC $\N^n$-set, then $\min(A)+\N^n=A$. The definition of $\min(A)$ implies that it is an antichain with respect to the weak order on $\Z^n$.   

There is a bijection between monomials in $k[x_1, \dots, x_n]$ and vectors in $\N^n$. If $I=<m_1, \dots, m_s>$, where $m_i=X^{a_i}$, then the monomials in $I$ are exactly the vectors in the $\N^n$-set generated by $A=\{a_1, \dots, a_s\}$.  

\begin{definition}
For $\alpha\in\Z^n$, the support of $\alpha$ is $\supp(\alpha)=\{i\mid\pi_i(\alpha)\neq0\}$.
\end{definition}

\begin{definition}
Let $\eta\in\Z^n$, and let $[n]=\{1, \dots, n\}$.  Let $T_{\eta}=\eta-\N^n=\{\eta-\alpha\mid\alpha\in\N^n\}$, and say that for nonempty $X\subseteq[n]$, an $X$-face of $T_{\eta}$ is $\{\alpha\in\Z^n | \pi_i(\alpha)=\pi_i(\eta) \text{ for all } i\in X\}$.  Let $T^o_{\eta}=\eta-\N^n_{>0}$. 
\end{definition}
 
\begin{definition}\label{generic}
Let $A\subseteq \Z^n$.  We say $A$ is generic if for all $\eta\in\Z^n$, such that $T^o_{\eta}\cap A=\emptyset$, $T_{\eta}$ contains at most one element of $A$ on each face.
\end{definition}

If $A$ is an $\N^n$-set that has a minimal element, it is never generic.  This is because if $\alpha\in\min(A)$, then $T^o_{\alpha+(1,0,\dots,0)}\cap A=\emptyset$, but $T^o_{\alpha+(1,0,\dots,0)}$ contains two points on one face. Because of this, we will adopt the convention of calling a DCC $\N^n$-set generic if its generating antichain is generic.

\subsection{Neighborly Sets}

If $A\subset\Z^n$, we wish to have a way of distinguishing certain subsets of $A$ that have desirable properties.  This distinction will be in the form of neighborly sets.

\begin{definition}
%Let $A\subset\Z^n$ and let $B\subset A$.  We say that $B$ is neighborly if for all $B'\subset A$ such that $\bigvee B=\bigvee B'$, $B=B'$.
Let $A\subset\Z^n$, and let $B\subset A$.  We say that $B$ is neighborly in $A$ if $T^o_{\vee B}\cap A=\emptyset$.  We say $B$ is maximally neighborly if $B$ is neighborly and $B'\supset B$ implies $B'$ is not neighborly. 
\end{definition}

\begin{example}
\
\begin{enumerate}
\item If $A\subset\Z^n$ is an antichain, then each $\alpha\in A$ is a neighborly set.
\item If $\Lambda \in\Z^2$ is generated by $(1,-1)$, and $A=\{(1,1)\}+\Lambda$, then $\{(i+1,i-1),(i+2,i-2)\}$ is a maximally neighborly set of $A$.
\item The empty set.
\end{enumerate}
\end{example}

\begin{lemma}\label{inducedneighborly}
If $A\subseteq\Z^n$, and $B\subseteq A$ is neighborly, then every subset of $B$ is neighborly.
\end{lemma}

\begin{proof}$ $

Since $B$ is neighborly, we have that $T^o_{\vee B}\cap A=\emptyset$.  Additionally, since $B'\subseteq B$, we have that $T^o_{\vee B'}\cap A\subseteq T^o_{\vee B}\cap A=\emptyset$, and hence $T^o_{\vee B'}\cap A=\emptyset$.  Therefore, $B'$ is neighborly.
\end{proof}

\begin{definition}
Let $A\subset\Z^n$ and let $B\subset A$.  If $B'\subseteq A$ and $\vee B'=\vee B$ implies that $B'=B$ for all such $B'\subseteq A$, then $B$ is called strongly neighborly.
\end{definition}

\begin{prop}\label{niceprop}
Let $A\subset\Z^n$.  Then $B\subset A$ strongly neighborly implies that $B$ is neighborly, and the converse holds if $A$ is generic.
\end{prop}

\begin{proof}$ $

Let $B$ be strongly neighborly.  Then for any $B'\subset A$ such that $\bigvee B=\bigvee B'$, we have that $B=B'$.  If $T^o_{\vee B}\cap A\neq\emptyset$, then there exists $\alpha\in A$ such that $\alpha<<\bigvee B$, and hence $\bigvee B=\bigvee(B\cup\alpha)$.  Then $B=B\cup\alpha$, which is a contradiction, and hence $T^o_{\vee B}\cap A=\emptyset$, so $B$ is neighborly.

Now suppose that $B$ is neighborly and $A$ is generic, then at most one element of $A$ lies on each face of $T_{\vee B}$ by definition.  Now consider $B'$ such that $\bigvee B=\bigvee B'$.  Each $\beta\in B$ contributes to $\bigvee B$ in some component because of genericity.  If $\beta'\in B'$ contributes to $\bigvee B$ what $\beta$ did, then they lie in the same face of $T_{\vee B}$ and hence must be the same.  In this manner, we conclude that each element of $B$ matches up with an element of $B'$, and vice versa, and hence $B=B'$, so $B$ is strongly neighborly.
 
\end{proof}

\begin{definition}\label{scarfdef}
If $A\subseteq\Z^n$, let $N(A):=\{\text{strongly neighborly sets of } A\}$, and let $N_i(A):=\{\sigma\in N(A) | |\sigma|=i+1\}$. We call $N(A)$ the Scarf complex of $A$.
\end{definition}

\begin{prop}\label{prop2.15}
 If $A\subseteq\Z^n$, then $N(A)$ is a simplicial complex.
\end{prop}

\begin{proof}$ $

By Lemma \ref{inducedneighborly}, neighborliness is closed under taking subsets.  Hence, $\sigma\in N_i(A)$ is an $i$-face of $N(A)$, and $N_{i-1}(A)\ni\tau\subseteq\sigma$ is a face of $\sigma$.
\end{proof}

\section{Antichain Lattices}

In the existing literature, the requirement that $\Lambda\cap\N^n=0$ is often imposed on lattices $\Lambda\subseteq\Z^n$.  For brevity, we will work with lattices that are also antichains.

If $\Lambda\subseteq\Z^n$ is an antichain lattice, then we define $\IL\subset k[x_1,\dots,x_n]$ to be the ideal generated by $$\{X^{\lambda^+}-X^{\lambda^-}\mid\lambda\in\Lambda\}$$  Notice that any monoid morphism $\phi:\N^n\rightarrow\N^m$ extends to a group homomorphism $\overline{\phi}:\Z^n\rightarrow\Z^m$, and that $\ker(\overline{\phi})$ is an antichain lattice.  Also, $\phi$ induces $\hat{\phi}:k[x_1,\dots,x_n]\rightarrow k[y_1,\dots,y_m]$, and $\ker(\hat{\phi})=I_{\ker(\overline{\phi})}=\{X^{\alpha^+}-X^{\alpha^-} | \alpha\in\ker{\overline{\phi}}\}$.

\subsection{Markov Bases} $ $

A Markov basis is a useful tool in bridging the gap between the combinatorial Scarf complex and the algebraic object $I_{\Lambda}$.  This will be done via the fundamental theorem of Markov bases (Theorem \ref{markdef}). Save for Proposition \ref{markovneighbors}, the basic Markov basis theory treatment is from \cite{DS}.

Consider an antichain lattice $\Lambda \subseteq \Z^n$.  Define the \emph{fiber over u} for $u \in \N^n$ to be $\mathcal{F}(u):= (u+\Lambda)\cap\N^n   = \{v\in \N^n | u-v\in\Lambda\}$.  Now consider an arbitary finite subset $\mathcal{B}\subseteq\Lambda$.  For an arbitrary element $u\in\N^n$, we can define a graph denoted $\mathcal{F}(u)_{\mathcal{B}}$ where the vertices are the elements of $\mathcal{F}(u)$ and the edges are between vertices $v, w$ if $v - w$ or $w - v$ are in $\mathcal{B}$.

\begin{definition}\label{markdef1}A Markov basis of a lattice $\Lambda\in\Z^n$ is a finite set $\mathcal{B}\subseteq\Lambda$ such that $\mathcal{F}_{\mathcal{B}}(u)$ is connected for all $u\in\N^n$. We call a Markov basis \emph{minimal} if it is such with respect to inclusion.
\end{definition}

\begin{theorem}\emph{[Theorem 1.3.2, \cite{DS}]}If $\mathcal{B}$ and $\mathcal{B}'$ are minimal Markov bases for a lattice, then $|\mathcal{B}|=|\mathcal{B}'|$.
\end{theorem}

\begin{theorem}\label{markdef}\emph{[Theorem 1.3.6, \cite{DS}]}A subset $\mathcal{B}$ of a lattice $\Lambda$ is a (minimal) Markov basis if and only if the set $\{X^{b^+}-X^{b^-} | b\in\mathcal{B}\}\subset k[x_1, \dots, x_n]$ forms a (minimal) generating set of the lattice ideal $\IL=<X^{b^+}-X^{b^-} | b\in\Lambda>$.
\end{theorem}

In the future, we will be referring to Theorem \ref{markdef} more often than to Definition \ref{markdef1}

\begin{definition}
Let $\Lambda\subset\Z^n$ be a lattice.  For any $\beta\in\Z^n$, the fiber over $\beta$ is $\beta+\N^n\cap\Lambda$. 
\end{definition}

\begin{prop}\label{markovneighbors}
Let $\Lambda\subseteq\Z^n$ be a lattice that is an antichain. If $B$ is a Markov basis of $\Lambda$ and $\mathcal{N}$ is the set of neighbors of the origin, then $\mathcal{N}=B\cup-B$.
\end{prop}

\begin{proof} $ $

First, notice that $\mathcal{N}\subseteq B\cup -B$ because if $\lambda_1$ and $\lambda_2$ are neighborly, then there is a fiber of $\Lambda$ that contains only $\lambda_1$ and $\lambda_2$.

For the opposite inclusion, it suffices to show that $\mathcal{N}$ is a Markov basis.  As a Markov basis, it will contain a minimal Markov basis, and because neighborliness is closed under taking negatives, it will also contain the negative of that minimal Markov basis.  For any two minimal Markov bases, $B$ and $B'$, it is the case that $B\cup -B=B'\cup-B'$, so we will be finished.  We proceed by proving that $\mathcal{N}$ is a Markov basis by showing that for any fiber, any two points in the fiber are connected by a path of neighborly pairs of elements.

Suppose that $F$ is a fiber of $\Lambda$ that contains only two elements.  Then those two elements are neighborly, and hence there is a neighborly path between them.  Now suppose that the result holds for all fibers $F$ such that $|F|<m$.  Suppose $F$ is a fiber such that $|F|=m$, and suppose $\lambda_1,\lambda_2\in F$ where $\lambda_1$ and $\lambda_2$ are not neighborly.  Without loss of generality, let $F$ be the fiber over $\lambda_1\wedge\lambda_2$.  Since $\lambda_1$ and $\lambda_2$ are not neighborly, there exists $\alpha\in F$ such that $\alpha<<\lambda_1\vee\lambda_2$.  

Let $$\Delta_1=\{i\in[1,\dots,n]\mid\pi_i(\lambda_1)>\pi_i(\lambda_2)\}$$ Then $\pi_i(\lambda_1)>\pi_i(\alpha)$ for all $i\in\Delta_1$ and $\pi_j(\lambda_1)<\pi_j(\alpha)$ for all $j\in\Delta_1^c$.  By construction, $\pi_i(\alpha)>\pi_i(\lambda_2)$ for all $i\in\Delta_1$, and $\pi_j(\alpha)<\pi_j(\lambda_2)$ for all $j\in\Delta_1^c$.  Therefore, $(\alpha-\lambda_1)\wedge 0>(\lambda_2-\lambda_1)\wedge 0$ and hence $(\alpha\wedge\lambda_1)>(\lambda_1\wedge\lambda_2)$.  

We can draw two conclusions from this final inequality.  The first is that $(\alpha\wedge\lambda_1+\N^n)\cap\Lambda\subset(\lambda_1\wedge\lambda_2+\N^n)\cap\Lambda$, and the second is that $\lambda_2\notin(\alpha\wedge\lambda_1+\N^n)\cap\Lambda$.  The final conclusion to draw is that the minimal fiber containing $\alpha$ and $\lambda_1$ has size less than $n$, and likewise for $\lambda_2$.  Thus, by the inductive hypothesis, there is a neighborly path from $\lambda_1$ to $\alpha$ and another from $\alpha$ to $\lambda_2$, creating the desired neighborly path from $\lambda_1$ to $\lambda_2$.

\end{proof}

Our use of Markov bases will be ubiquitous henceforth.  The primary goal of this section was to establish the fact that the generating sets of the ideals we will work with later all have a very specific form.  More structural lemmas along these lines will establish this fact more rigorously later.

\subsection{Generic Lattices}

In our quest to unite the various definitions of genericity, we will now consolidate two definitions of generic from the literature. Namely, we will unite Definition \ref{genlat} from \cite{PS} and Definition \ref{generic} from \cite{MS}.  

\begin{definition}\label{genlat}If $\Lambda\subset\Z^n$ is an antichain lattice, we say $\Lambda$ is generic if there is a minimal Markov basis $L$ of $\Lambda$ such that each $\lambda\in L$ is fully supported.
 \end{definition}
 
 \begin{lemma}\label{genericsmatch}
If $\Lambda\subset\Z^n$ is an antichain lattice, then $\Lambda$ is generic as in Definition \ref{genlat} if and only if $\Lambda$ is generic in $\Z^n$ as in Definition \ref{generic}.
\end{lemma}

\begin{proof} $ $

By Proposition \ref{markovneighbors}, we can first consider an identical statement: the neighbors of the origin with respect to $\Lambda$ are fully supported if and only if there are no neighborly pairs that share a component.

Let $\Lambda$ be generic by Definition \ref{genlat}. Under lattice translations, if $$L=\{\text{neighbors of the origin with respect to }\Lambda\},$$ then $$\alpha+L=\{\text{neighbors of }\alpha\text{ with respect to }\Lambda\}$$  If $\beta\in\alpha+L$, then $pi_1(\alpha)\neq\pi_i(\beta)$ for $i=1,\dots, n$ because the elements of $L$ are fully supported, and $beta=\alpha+\ell$ for some $\ell\in L$. Because of this, if there exists a $\beta$ such that $\pi_i(\beta)=\pi_i(\alpha)$, then $\alpha$ and $\beta$ are not neighborly.  Therefore, there exists $\gamma\in T^o_{\alpha\vee\beta}\cap\Lambda$ by definition.  That is, there exists $\gamma<<\alpha\vee\beta$ and hence $\Lambda$ is generic by Definition \ref{generic}.

Let $\Lambda$ be generic by Definition \ref{generic}.  Then for all $\alpha,\beta\in\Lambda$ such that $pi_i(\alpha)=\pi_i(\beta)$, there exists $\gamma\in\Lambda$ such that $\gamma<<\alpha\vee\beta$.  That is, $\gamma\in T^o_{\alpha\vee\beta}\cap\Lambda$.  Therefore, if $\pi_i(\alpha)=\pi_i(\beta)$ for some $i=1,\dots, n$, then they are not neighborly.  Hence, if $\alpha,\beta$ are to be neighborly, $\alpha-\beta$ must be fully supported. Thus, if $A$ is the set of neighbors of $\alpha$, then the vectors $\{\alpha-\beta\mid\beta\in A\}$ are fully supported, and hence $\Lambda$ is generic by Definition \ref{genlat}.
\end{proof}

Lemma \ref{genericsmatch} shows us that the notion of a generic lattice from \cite{PS} matches the definition for generic we have already seen for $\N^n$-sets.

\section{$\Lambda$-sets}

In this section, we will generalize the lattices from the previous section into $\Lambda$-sets, and then reform some of the notions and definitions we had for lattices.  If not explicitly mentioned, our lattices will continue to be subsets of $\Z^n$, antichains and generic. 

The primary object of study in this section is a $\Lambda$-set, which is a specific case of an $M$-set, where $M$ is a monoid. If $A\subseteq\Z^n$, and $A=A+\Lambda$, then $A$ is a $\Lambda$-set under the map $A\times\Lambda\rightarrow A$ defined by $(\alpha,\lambda)\mapsto\alpha+\lambda$.

\subsection{Structure of $\Lambda$-sets}
\begin{definition}
 Suppose $A=A+\Lambda$. If $A_0\subseteq A$, we call $A_0$ a set of $\Lambda$-representatives for $A$ if
\begin{enumerate}
\item $A=A_0+\Lambda$
\item $a,b\in A_0$ implies $a-b\notin\Lambda$
\end{enumerate}
Call $A$ $\Lambda$-finite if $A$ has a finite set of representatives.
\end{definition}

\begin{remark}
All $\Lambda$-finite sets are DCC sets, a fact that will be used nearly constantly without mention.
\end{remark}

Unless $\Lambda=\{0\}$, infinitely many options for $A_0$ exist. When thinking of $A=\Lambda\cup(\alpha_0+\Lambda)$, we could choose $A_0=\{\alpha,\alpha_0+\beta\}$ for any $\alpha,\beta\in\Lambda$ without any reference to the Euclidean distance between $\alpha$ and $\beta$. It will be important later to be able to address this distance, so we will develop a method for choosing an $A_0$ that has an additional desirable property: closeness.

\begin{lemma}\label{closeness}
Let $\Lambda\subset\R^n$ and let $A$ be $\Lambda$-finite.  Let $V$ be the subspace of $\R^n$ spanned by $\Lambda$, and let $\mathcal{C}$ be a fundamental region ($k$-parallelapiped, where $\Lambda$ has codimension $n-k$) of $\Lambda$ in $V$.  If $\pi:\R^n\rightarrow V$ is the orthogonal projection map, then there is a set of $\Lambda$-representatives for $A$ contained in $\pi^{-1}(\mathcal{C})$.
\end{lemma}

\begin{proof}
We have that $A$ is $\Lambda$-finite, so choose $A_0$ as a finite set of representatives.  For ease, order $A_0$ as $\alpha_1 < \alpha_2<\cdots<\alpha_s$, and consider $\pi(\alpha_1)+\mathcal{C}$. Since $\pi(\alpha_1)\neq\pi(\alpha_i)$ for all $i>1$, and $\pi(\alpha_i)+\mathcal{C}+\Lambda$ is a division of $V$ into $k$-parallelapipeds, there exists a $\lambda_i\in\Lambda$ such that $(\pi(\alpha_i)+\mathcal{C}+\lambda_i)\cap(\pi(\alpha_i)+\mathcal{C})\neq\emptyset$.  To complete the proof, let the representative set be $\pi^{-1}(\alpha_1)\cup\{\pi^{-1}(\alpha_i+\lambda_i) | i>1\}$.
\end{proof}

Although there will be many situations where this property is not needed, we will henceforth only consider sets of $\Lambda$-representatives of $\Lambda$-finite sets of the form of the conclusion of Lemma \ref{closeness}.   

\begin{prop}\label{structural}Let $A$ be a generic $\Lambda$-finite set, then $N_i(A)$ is $\Lambda$-finite set under the map $N_i(A)\times\Lambda\rightarrow N_i(A)$ where $(\sigma,\lambda)\mapsto\sigma+\lambda$
\end{prop}

\begin{proof}$ $
 If $\sigma\in N_i(A)$, then $\sigma+\lambda\in N_i(A)$ for all $\lambda\in\Lambda$, so $N_i(A) = N_i(A)+\Lambda$, and hence it is a $\Lambda$-set.  The $\Lambda$-finiteness property will come as a corollary to Lemma \ref{locallyfinitelemma}.
\end{proof}

\begin{comment}
We will conclude this section with a few remarks about the generalizations we have done here.  First, note that if $A_0$ is a singleton, then $A$ is a translation of $\Lambda$.  Additionally, the simplicial complex put on $\Lambda$ is identical to the simplicial complex that we can define iteratively by saying two points $\alpha,\beta\in\Lambda$ have an edge between them there is no third element $\gamma\in\Lambda$ such that $\gamma<<\alpha\vee\beta$.  This coincides verbatim with the definition of the Scarf complex (Lemma \ref{scarfneighbors}), which in turn coincides with the Buchberger graph from \cite{MS}, and shows the reader that we are augmenting the space of objects that previous algorithms have been applied to.

Furthermore, since we now have a simplicial complex, we are granted the existence of certain mappings involving the faces of subsimplices; these maps will be exploited in great detail in later chapters.
\end{comment}

\section{Resolutions}

This section will review our primary object of study: resolutions.  We will mostly address the general definitions via our specific uses, and in particular, via a constructive algorithm.  We will cover the definitions associated to cellular resolutions, which encompasses the algorithm that we will apply to the scarf complex in later chapters.

 \begin{definition}
 Let $M$ be an $S$-module, then a \emph{resolution} of $M$ is a complex $F_{\bullet}$ with maps $\delta_i$ such that $$0\longleftarrow M \overset{\delta_0}\longleftarrow F_0 \overset{\delta_1}\longleftarrow F_1 \overset{\delta_2}\longleftarrow \cdots \leftarrow :F_{\bullet}$$ is exact. I.e., if $\ker(\delta_i)=\im(\delta_{i+1})$. The resolution is free if $F_i$ is free for all $i$. If the resolution is free, then $F_i=S^{\beta_i}:=\underbrace{S\oplus \cdots \oplus S}_{\beta_i\text{ times}}$, and if it is minimal, the $\beta_i$s are collectively called the Betti numbers of the resolution.
\end{definition}

\subsection{Resolutions of Lattice Ideals}

Later, we will cover resolutions of lattice ideals in more generality, but for this section, we will give the basic results concerning lattice ideals. 

\begin{definition}\label{M_A}\emph{[Definition 9.11, \cite{MS}]} Let $A\subseteq\Z^n$.  Then $M_A$, is the $S$-submodule of the Laurent polynomial ring $S^{\pm}=k[x_1^{\pm 1}, \dots, x_n^{\pm 1}]$ generated by $\{X^{\alpha} | \alpha\in A\}$.
\end{definition}

In \cite{MS}, one will find that the Scarf complex of $A\in\Z^n$ is defined as the set of strongly neighborly sets, where we have defined it to be the set of neighborly sets.  We saw in Lemma \ref{niceprop} that when $A$ is generic, strongly neighborly and neighborly are identical, and as such, the reader does not need to make any distinction going foward.

We will finish this section with a prelude to what we intend to do with the machinery we have hitherto developed.  In section \ref{next}, we will construct a collection of maps that we will associate to simplicial complexes.  When we apply this construction to the Scarf complex of a generic $\Lambda$-set, $A$, we will obtain free a free resolution of $M_A$ as an $S$-module.  Additionally, we will be able to resolve lattice ideals by considering the construction modulo the lattice.  The machinery behind these ideas will be developed in later sections in more general situations. The machinery will primarily exploit the structure of the lattice, and in fact, we will use a more general version of the Scarf complex. 

\subsubsection{Lattice Ideal Resolutions in $\Z^3$}

In $\Z^3$, we have a remarkable amount of control over Markov bases of lattices.  In particular, the Markov bases will have three elements, $\lambda_1, \lambda_2$, and $\lambda_3$, and they can be chosen such that $\lambda_1=-(\lambda_2+\lambda_3)$.

\begin{lemma}\label{lambdares}
If $\Lambda\subset\Z^3$ is a generic antichain lattice with codimension 1 and Markov basis $\lambda_1 = \{(\alpha_1,-\beta_1,-\gamma_1), \lambda_2 = (\alpha_2,-\beta_2,-\gamma_2), \lambda_3 = (\alpha_3,-\beta_3,-\gamma_3)\}$, then the minimal free resolution of $S/\IL$ is 
\begin{displaymath}
\begin{array}{ccccccc}
S/\IL & \leftarrow & S & \leftarrow & Se_{\lambda_1}\oplus Se_{\lambda_2}\oplus Se_{\lambda_3} & \leftarrow & Se_{p_1}\oplus Se_{p_2} \\ 
 &  & b_1 & \mapsfrom & e_{\lambda_1} & x_3^{\gamma_2}e_{\lambda_1}+x_1^{\alpha_3}e_{\lambda_2}+x_2^{\beta_1}e_{\lambda_3} & \mapsfrom e_{p_1} \\ 
 &  & b_2 & \mapsfrom & e_{\lambda_2} & x_2^{\beta_3}e_{\lambda_1}+x_3^{\gamma_1}e_{\lambda_2}+x_1^{\alpha_2}e_{\lambda_3} & \mapsfrom e_{p_2} \\ 
 &  & b_3 & \mapsfrom & e_{\lambda_3} &  & 
\end{array}
\end{displaymath}
\end{lemma}

\begin{proof}
Apply the tools from section \ref{6} that we will cover latter.  Alternatively, \cite{H3}.
\end{proof}

\subsection{Cellular Resolutions}\label{next}

Let $\Lambda\subseteq\Z^n$ be an antichain lattice, and let $A$ be a generic $\Lambda$-finite set.  We already have that $N(A)$ is a simplicial complex; to the simplicial structure, we can add more information in the form of face labels. We will label the face $\sigma$ of $N(A)$ with $\vee\sigma$.

\begin{definition}
 Let $S=k[x_1, \dots, x_n]$, and let $F_i(N(A)):=\displaystyle{\bigoplus_{\sigma\in\N_i(A)}Se_{\sigma}}$ be the free $S$-module with generators $\{e_{\sigma}\mid\sigma\in N_i(A)\}$. 
\end{definition}

%We wish to define maps from $F_i(N(A))$ to $F_{i-1}(N(A))$, but before we do that, we must put an orientation on $N(A)$.  First let $N_k(A)$ be the highest dimensional nonempty collection of neighborly sets.  We will assign an arbitrary orientation to $N_k(A)$, and use that orientation to induce orientations on all smaller neighborly sets.  To do this, we will say that $\sigma\in N_k(A)$ has the form $\sigma=\{v_0,\dots,v_k\}$, and that the faces of $\sigma$ will be denoted by $\sigma\setminus i:=\{v_0,\dots,v_{i-1},v_{i+1},\dots,v_k\}=\{w_0,\dots,\w_{k-1}\}$.  We continue labeling iteratively, and say that a face $\sigma\setminus i$ of $\sigma$ inherits the orientation from $\sigma$ is $i$ is even, and has opposite orientation if $i$ is odd. 

If $\sigma=\{\sigma_0,\dots,\sigma_i\}\in N_i(A)$, then $\partial_j\sigma=\{\sigma_0,\dots,\sigma_{j-1},\sigma_{j+1},\dots,\sigma_i\}$.  Let $\phi_i: F_i(N(A))\rightarrow F_{i-1}(N(A))$ be defined as follows: \begin{equation}\label{mapeq}\begin{array}{cccc}\phi_i: & F_i(N(A)) & \rightarrow & F_{i-1}(N(A)) \\ & e_{\sigma} & \mapsto & \sum_{j=0}^i(-1)^jX^{\vee\sigma-\vee\partial_j\sigma}e_{\partial_j\sigma}\end{array}\end{equation}

\begin{prop}
 With $\phi_i$ defined above, $\phi_i\phi_{i-1}=0$.
\end{prop}

\begin{proof}
Chapter 8 of \cite{W1}.
\end{proof}

\begin{definition}
Let $X$ be a simplicial complex labeled with suprema in $\Z^n$, and let $X_i$ be the set of $i$-faces of $X$.  The cellular free complex supported on $X$, denoted $\mathcal{F}_{X}$, is the complex of free $k[x^{\pm 1}_1, \dots, x^{\pm 1}_n]$-modules generated by $e_{\sigma}$ for $\sigma\in X_i$. If $X$ is acyclic, we pair $X$ together with the maps $\phi$ from (\ref{mapeq}) to obtain the cellular free resolution supported on $X$.  We also denote it $\mathcal{F}_{X}$.
\end{definition}

\begin{definition}
If $X$ is a simplicial complex labeled with elements of $\Z^n$, then for all $b\in\Z^n$, $X_{\preceq b}$ is the subcomplex supported on all faces $\sigma$ such that $\vee\sigma\leq b$.
\end{definition}

\begin{prop}
The cellular free complex $\mathcal{F}_{X}$ supported on $X$ is is a cellular resolution if and only if $X_{\preceq b}$ is acyclic over $k$ for all $b\in\Z^n$.  When $\mathcal{F}_{X}$ is acyclic, then it is a free resolution of $M_{X}=\{X^{\zeta}\mid\zeta\text{ the label of some face of } X\}$, the $k[x_1,\dots,x_n]$-submodule of $k[x_1^{\pm1}, \dots, x_n^{\pm1}]$.
\end{prop}

\begin{proof}$ $

This is an extension of the finite case given in Proposition 4.5 in \cite{MS}, but the proof runs identically.
\end{proof}

\begin{example}\label{isresolution}
Let $\Lambda\subset\Z^n$ be an antichain lattice, and let $A$ be a generic $\Lambda$-finite set. Then $$F_{\bullet}: \cdots \rightarrow F_i(N(A))\overset{\phi_i}\rightarrow F_{i-1}(N(A))\overset{\phi_{i-1}}\rightarrow\cdots F_0(N(A))\overset{\phi_1}\rightarrow M_A$$ is a resolution of $M_A$ as an $S$-module.
\end{example}

\subsection{Taylor and Hull Resolutions}

\subsubsection{Hull Complex}

We begin with some notation.  We will always assume that $t\in\R$ with $t>1$ and that $A\subset\Z^n$.  Let $$E_t(\alpha)=(t^{\pi_1(\alpha)}, \dots, t^{\pi_n(\alpha)})$$ for $\alpha\in\Z^n$ and $$E_t(A)=\{E_t(\alpha)\mid\alpha\in A\}$$  Additionally, we will let $$\mathcal{P}_t(A)=\conv(E_t(A)+\N^n)= \R_{\geq0}^n+\conv(E_t(A))$$

\begin{lemma}\label{exponentialconvex}
If $A\subseteq\Z^n$ is a generic $\Lambda$-finite set for some antichain lattice $\Lambda\subseteq\Z^n$, then for $t>1$, the vertices of $\mathcal{P}_t(A)$ are $E_t(A)$.
\end{lemma}

\begin{proof}$ $

It suffices to show that $E_t(A)$ is convex for large enough $t$.  First note that from \cite{RK}, we have the following condition for convexity: a set $C\in\R^n$ is convex if and only if for all $x,y\in C$, $<N_C(x)-N_C(y), x-y>\geq 0$, where $N_C(x)$ is the normal vector to $C$ at $x$.\footnote{In \cite{RK}, as here, we will consider a normal vector at a point to be any vector inside the normal cone at that point.  That is, we can choose a normal vector to any plane that is tangent at the point, and the result still hols.}

Let $a\in\R_{>0}^n$ and let $t\in\R_{>1}$.  Let $$H_a=\{x\in\R^n | a\cdot x\geq 0\}$$ 

and 

$$\partial H_a=\{x\in\R^n | a\cdot x=0\}.$$

Then $$t(H_a)=\{t(x) | a\cdot x \geq 0\}$$
$$=\{(t^{x1}, \dots, t^{x_n}) | a_1x_1 + \cdots + a_nx_n\geq 0\}$$
$$=\{(\xi_1, \dots, \xi_n) | \xi_1^{a_1}\dots\xi_n^{a_n}\geq 1, \xi_i=t^{x_i}\}$$

and 

$$t(\partial H_a)=\{(\xi_1, \dots, \xi_n) | \xi_1^{a_1}\dots\xi_n^{a_n} = 1\}$$

To simplify notation, let $f_a(\xi)=\xi_1^{a_1}\dots\xi_n^{a_n}$.  Then we have that $t(\partial H_a)$ is the level set defined by $f_a(\xi)=1$ and $t(H_a)=\{\xi | f_a(\xi)\geq 1\}$.  Note that since $t$ is a homeomorphism from $\R^n$ to $\R^n$, we have that $t(\partial H_a)=\partial t(H_a)$.

We wish to show that $C=t(H_a)$ is convex.  By the aforementioned convexity condition, we can show $<N_C(x)-N_C(y),x-y>\geq 0$ for all $x,y\in C$.  We clearly only need to check this on the boundary of $C$, which is what we will do.  The (outward facing) normal vector to $\partial C$ at $\xi$ is $-\nabla f_a(\xi)$.  Now $\frac{\partial f_a}{\partial \xi_i}=\frac{a_i}{\xi_i}f_a(\xi)$ and if $\xi\in\partial C$, then $f_a(\xi)=1$.  Thus $\nabla f(\xi)=(\frac{a_1}{\xi_1},\cdots, \frac{a_n}{\xi_n})$ for all $\xi\in\partial C$.  

To finish the computation, choose $\xi,\eta\in\partial C$.  Then $$N_C(\xi)=-(\frac{a_1}{\xi_1},\cdots, \frac{a_n}{\xi_n})$$
and
$$N_C(\eta)=(\frac{a_1}{\eta_1},\cdots, \frac{a_n}{\eta_n})$$

Now $<N_C(\xi)-N_C(\eta),\xi-\eta>=<(\dots, \frac{a_i(\xi_i-\eta_i)}{\xi_i\eta_i},\dots),(\dots,\xi_i-\eta_i,\dots)>$
$=a_1\frac{(\xi_1-\eta_1)^2}{\xi_1\eta_1} + \cdots + a_n\frac{(\xi_n-\eta_n)^2}{\xi_n\eta_n}\geq 0$.  So we have that $C$ is convex.
\end{proof}

\begin{corollary}\label{cor5.2}
Let $A\subseteq\Z^n$ be a generic $\Lambda$-finite set for some antichain lattice $\Lambda\subseteq\Z^n$, and $t>1$.  If $F$ is a face of $\conv(E_t(A))$, then $F\cap E_t(A)=E_t(\sigma)$ where $\sigma\in N(A)$.
\end{corollary}

\begin{proof}

We already have that $E_t(A)$ is the vertex set of $\mathcal{P}_t(A)$.  Suppose that $F$ is a maximal face of $E_t(A)$ and let $F\cap E_t(A)=\{t^{\alpha_1}, \dots, t^{\alpha_r}\}$.  Suppose for a contradiction that $\{\alpha_1, \dots, \alpha_r\}\notin N(A)$.  Then there exists $b\in A$ such that $b<<\vee\alpha_i$.  Therefore $E_t(b)<<E_t(\vee\alpha_i)=\vee E_t(\alpha_i)$.  We have three cases to consider.
\begin{enumerate}
\item $E_t(b)\in\conv(E_t(\alpha_1), \dots, E_t(\alpha_r),\vee E_t(\alpha_i))$.
\item $E_t(b)\notin\conv(E_t(\alpha_1), \dots, E_t(\alpha_r),\vee E_t(\alpha_i))$.
\item $E_t(b)\in F$.
\end{enumerate} 

Examining each case:
\begin{enumerate}
\item We would have that $E_t(b)$ lies in the interior of $\mathcal{P}_t(A)$, contradicting Lemma \ref{exponentialconvex}.
\item This would imply that the hyperplane containing $F$ separates $E_t(A)$, contradicting the convexity of $\mathcal{P}_t(A)$.
\item If $E_t(b)\in F$, increase $t$ by $\epsilon> 0$ to be back in case 2.
\end{enumerate}
\end{proof}

Before we cover the main concepts in this section, we first need a structural lemma that underlies many statements that will be made later.

\begin{lemma}
 Let $A\subset\Z^n$, and suppose that $A$ is a generic $\Lambda$-finite set for some antichain lattice $\Lambda$.  Then for $t>>0$, $\partial(\conv(E_t(A)+\N^n))\cong\R^{n-1}$.
\end{lemma}

\begin{proof} $ $

Let $B=\{\beta\in\R^n\mid\pi_1(\beta)+\cdots+\pi_n(\beta)=0\}$, and for each $\beta\in B$, let $\ell_{\beta}=\{\beta+s(1,\dots,1)\mid s\in\R\}$.  Since $\partial(\conv(E_t(A)+\N^n))$ is convex, and $B\cap\N^n=0$, we have that each $\ell_{\beta}$ intersects $\partial(\conv(E_t(A)+\N^n))$ in at most one point. To see that $\ell_{\beta}$ intersects $\partial(\conv(E_t(A)+\N^n))$ at all, notice that the point of $\partial(\conv(E_t(A)+\N^n))$ that is closest to the origin is the point of intersection with $\ell_0$.  Call this point $\gamma$.  Then $\N^n\subset\partial(\conv(E_t(A)+\N^n))-\gamma$.  The line connecting any point $\eta$ on any coordinate face of $\N^n$ to the closest point on $B$ passes through $\partial(\conv(E_t(A)+\N^n))-\gamma$, showing that each $\ell_{\beta}$ intersects $\partial(\conv(E_t(A)+\N^n))$ in exactly one point.

 Therefore, we have a bijection between $B$ and $\partial(\conv(E_t(A)+\N^n))$.  For each $\beta\in B$, call this point of intersection $\beta'$.  Consider the map $$\begin{array}{cccc} f: & \partial(\conv(E_t(A)+\N^n)) & \rightarrow & B \\ & \beta' & \mapsto & \beta\end{array}$$

Since $f$ maps different elements along lines parallel to $t(1,\dots,1)$, then two points that are close in $\partial(\conv(E_t(A)+\N^n))$ remain close under $f$.  This also holds mutatis mutandis under $f^{-1}$, which maps $\beta'$ to $\beta$.  Therefore, we have a continuous bijection with a continuous inverse, and hence $\partial(\conv(E_t(A)+\N^n))$ and $B$ are homeomorphic.  Since $B$ is a hyperplane in $\R^n$, it is homeomorphic to $\R^{n-1}$, and hence, so is $\partial(\conv(E_t(A)+\N^n))$.
\end{proof} 

Continuing, we need to show an important property of $A$.

\begin{lemma}\label{locallyfinitelemma}
Let $\Lambda$ be an antichain lattice, and let $A\in\Z^n$ be a generic $\Lambda$-finite set. Then for each $\alpha\in A$, the set of neighbors of $\alpha$ is finite.
\end{lemma}

\begin{proof}
The proof runs similarly to the proof of Proposition 9.4 in \cite{MS}.  Since $A$ is $\Lambda$-finite, we can choose a set of $\Lambda$-representatives and call it $A_0$.  Then we have $|A_0|$ copies of $\Lambda$ in $A$.  We can find all the primitive elements (defined in the referenced proof) by individually translating each copy of $\Lambda$ to contain the origin, finding the associated primitive elements, then translating them back.  There are only finitely many primitive elements for each copy of $\Lambda$, and hence only finitely many overall.  

The second half of the proof runs identically.
\end{proof}

For $A\subseteq\Z^n$, let $\hull_t(A)=\{E_t(F)\subseteq E_t(A)\mid \conv(E_t(F))\text{ is a face of }\mathcal{P}_t(A)\}$.

\begin{prop}\label{stableposet}
If $A\in\Z^n$ is generic, then there exists $T\in\R$ such that for $t,t'\geq T$, $\hull_t(A)=\hull_{t'}(A)$.
\end{prop}

\begin{proof}$ $

Let $B_i=B(0,i)$ be the ball of radius $i$ about the origin in $\R^n$.  If $\mathcal{V}_{i,t}=B_i\cap\hull_t(A)$, then $\hull_t(A)=\varinjlim\mathcal{V}_{i,t}$.  By Proposition 4.14 of \cite{MS}, there exists a $T\in\R$ such that for $t,t'\geq T$, $\hull_t(\mathcal{V}_{i,t})=\hull_{t'}(\mathcal{V}_{i,t})$.  Specifically, the Proposition tells us that $T=(n+1)!$.  Since this holds for all $\mathcal{V}_{i,t}$, it holds under the direct limit, and hence when $T>(n+1)!$, $\hull_t(A)=\hull_{t'}(A)$.
\end{proof}

\begin{remark}
Although not mentioned explicitly, if $A$ were not generic, Proposition \ref{stableposet} fails.  This is because there will exist two elements that share a component without a third element dividing the supremum of the first two.  Under the exponentiation, these two elements would continue to share a component for all $t$, which would imply the existence of a supporting hyperplane of $\mathcal{P}_t(A)$ that was parallel to a coordinate plane, violating Lemma \ref{exponentialconvex}.
\end{remark}

When $t$ is large enough, $\hull_t(A)$ is independent of $t$, so we will drop the subscript and use $\hull(A)$ when it is understood that $t\geq T$.

\begin{prop}\label{locallyfinite}
Let $A\subset\Z^n$ be a $\Lambda$-finite set for some antichain lattice $\Lambda\subset\Z^n$.  For all $\alpha\in A$, $$|\{\sigma\in\hull(A)\mid\alpha\in\sigma\}|<\infty$$
\end{prop}

\begin{proof}$ $

 If a face of $\hull(A)$ were incident with infinitely many other faces, that would imply the existence of an edge that was incident with infinitely many other edges; up to a suitable translation, we could consider the point of incidence to be 0, contradicting Lemma \ref{locallyfinitelemma}.
\end{proof}

\begin{remark}
 In Lemma \ref{locallyfinitelemma}, we worked strictly in $A$ and $N(A)$, but Proposition \ref{locallyfinite} made a claim about $\hull(A)$.  However, we have a structure-preserving bijection between the two objects, so, up to notation, the claim in the lemma could have been made as a claim about $\hull(A)$.
\end{remark}

\begin{prop}
If $A\subseteq\Z^n$ is a $\Lambda$-finite set for some antichain lattice $\Lambda\subseteq\Z^n$, then every face of $\conv(E_t(A))$ is a polyhedron.
\end{prop}

\begin{proof} $ $

It is clear that $\conv(E_t(A))$ is the intersection of half-spaces from Lemma \ref{exponentialconvex}, so it remains to show that each face is the convex hull of finitely many points.  
If $\conv(E_t(A))$ had a supporting hyperplane that contained infinitely many points, that would imply the existence of a hyperplane containing infinitely many points of $A$.  The only such hyperplanes are those that are parallel to $\Lambda$ and that contain $\alpha_0+\Lambda$ for some $\alpha_0\in A$.  But by Theorem 9.14 of \cite{MS}, these collections of points are mapped to locally finite sets under the exponentiation map, and hence no supporting hyperplane of $\conv(E_t(A))$ containing infinitely many points exists.
\end{proof}

\subsubsection{Taylor Complexes and Resolutions}

\begin{definition}
A simplicial complex with labels from a lattice is a function from the vertices of the complex to the lattice. The label of a simplex is the supremum of the labels of its vertices.
\end{definition}

\begin{definition}\label{cellrescoeffs}
Let $\Delta$ be a simplicial complex labeled with suprema from $\Z^n$, and let $\Delta_i=\{i\text{-faces of }\Delta\}$.  Let $S=k[x_1,\dots,x_n]$ and $S(e_{\sigma})$ be the principal $S$-module generated by $e_{\sigma}$. The Taylor Complex supported on $\Delta$ is $$\mathcal{F}_{\Delta}: \cdots \overset{d_{i+1}}\rightarrow\mathcal{F}_i\overset{d_i}\rightarrow\cdots\overset{d_1}\rightarrow\mathcal{F}_0\overset{d_0}\rightarrow 0$$ where $$\mathcal{F}_i=\bigoplus_{\sigma\in\Delta_i}S(e_{\sigma})$$ and if $\sigma=\{\alpha_0,\dots,\alpha_i\}$, and $\sigma\setminus j=\{\alpha_0,\dots,\alpha_{j-1},\alpha_{j+1},\dots,\alpha_i\}$, $$d(e_{\sigma})=\sum_{\alpha_j\in\sigma}(-1)^{j-1}(X^{\vee\sigma-\vee\sigma\setminus j})e_{\sigma\setminus j}.$$
\end{definition}

\begin{remark}
In \cite{MS}, the Taylor complex is defined on a finite set in $\N^n$, but there is no reason for this other than making the $\Delta_i$ finite.
\end{remark}

\begin{definition}
 The Taylor resolution of $A\subseteq \Z^n$ is the Taylor complex supported on the simplicial complex that is full over $A$.  I.e., the faces of the simplicial complex are in bijection with the finite subsets of $2^A$.
\end{definition}

\begin{definition}
If $N(A)$ is the Scarf complex of $A$ (Definition \ref{scarfdef}), then $\mathcal{F}_{N(A)}$ is the algebraic Scarf complex, which is the Taylor complex supported on the Scarf complex.
\end{definition}

\begin{remark}
Note that the Scarf complex is a labeled simplicial complex, and the algebraic Scarf complex is that complex coupled with a collection of maps.
\end{remark}

\begin{prop}\label{scarfsubcomplex}
If $A\subseteq\Z^n$, then every free $S$-resolution of $M_A$ contains the algebraic Scarf complex $\mathcal{F}_{N(A)}$ as a subcomplex.
\end{prop}

\begin{proof}$ $

The Taylor resolution is an $S$-resolution of $M_A$.  By \cite{P2}, it must contain a minimal resolution.  Call that minimal resolution $\mathcal{F}_{\bullet}$.  By definition, $\mathcal{F}_{\bullet}$ must contain all relations of $M_A$ in all dimensions.  Additionally, the Taylor resolution contains the Scarf complex by construction, which in turn contains relations of $M_A$ without repitition.  Since the Scarf complex does not necessarily contain all relations, it is a subcomplex of $\mathcal{F}_{\bullet}$.
\end{proof}

\begin{theorem}\label{theorem137}
If $A\subset\Z^n$, then $\mathcal{F}_{N(A)}$ is isomorphic to a subcomplex of $\hull(A)$.
 \end{theorem}

\begin{proof} $ $

Let $\sigma\subset A$ be a face of the Scarf complex.  Then $\sigma$ is strongly neighborly.  We wish to relabel the elements of $\sigma$ in a meaningful way.  To do this, consider $i\in [p]$ and let $$J(i)=\{j\in [n] \mid \pi_j(\vee\sigma\setminus i)<\pi_j(\vee\sigma)\}$$

Notice that $J(i)$ is nonempty, because if it was empty, then $\alpha_i$ would not contribute to $\bigvee \sigma$, and hence $\sigma$ could not be neighborly because $\bigvee\sigma = \bigvee(\sigma\setminus\alpha_i)$.  Additionally, for similar reasons, $J(i)\nsubseteq\bigcup_{k\neq i}J(k)$.  Therefore, for each $i\in [p]$, there is a $j=j(i)\in [n]$ such that $\alpha_i$ contributes to $\bigvee\sigma$ in component $j(i)$ and no other element of $\sigma$ does.  Now for each $\alpha_i\in\sigma$, choose such a $j(i)$, and relabel $\alpha_i$ as $\alpha_{j(i)}$.  Then $\pi_i(\alpha_i)>\pi_i(\alpha_k)$ for all $k\neq i$.

The second step of the proof is that $\{t^{\pi_i(\alpha_k)}\}$ is a nonsingular matrix for large enough $t$.  It suffices to show this by showing that for large enough $t$, \begin{equation}\label{tlarge}\prod_{i=1}^pt^{\pi_i(\alpha_i)}>p!\prod_{i=1}^pt^{\pi_i(\alpha_{\rho(i)})}\end{equation} for any non-identity permutation $\rho$ of $[p]$.  If (\ref{tlarge}) is satisfied, then the term $\prod_{i=1}^pt^{\pi_i(\alpha_i)}$ will dominate all other terms $\det(\{t^{\pi_i(\alpha_k)}\})$, and hence the matrix will be nonsingular. Assume $t>p$, then $$\frac{\prod_{i=1}^pt^{\pi_i(\alpha_i)}}{\prod_{i=1}^pt^{\pi_i(\alpha_{\rho(i)})}}=\prod_{i=1}^pt^{\pi_i(\alpha_i)-\pi_i(\alpha_{\rho(i)})}\geq\prod_{i=1}^pt\geq\prod_{i=1}^pp=p^p>p!$$

Therefore, inequality (\ref{tlarge}) is satisfied for all non-identity permutations $\rho$. 

This says that the points $\{t^{\alpha_1}, \dots, t^{\alpha_p}\}$ are affinely independent.  Because they are affinely independent, the convex hull of the points forms a simplex in which every point is a vertex.

By definition, $\hull(A)_{\preceq\vee\sigma}$ is exactly the convex hull of $\{t^{\alpha_1}, \dots, t^{\alpha_p}\}$.  Because $\sigma$ is (strongly) neighborly, there is no other subset of $A$ that has the same supremum as $\sigma$.  As such, if a face of $\hull(A)$ is labeled with $x^{\vee\sigma}$, it necessarily came from the image of $\sigma$, and since the exponential map is injective, there can be only one such face.  Proposition \ref{scarfsubcomplex} says that every free resolution contains the algebraic Scarf complex as a subcomplex.  This tells us that in addition to there being at most one face with label $x^{\vee\sigma}$, there also must be at least one.  Therefore, every strongly neighborly set of $A$ is present as a face in $\hull(A)$.
\end{proof}

\begin{remark}It will be common to drop the phrase "is isomorphic to" from Theorem \ref{theorem137} and just say that $\mathcal{F}_{N(A)}$ is a subcomplex of $\hull(A)$.
\end{remark}

\begin{theorem}\label{scarfequalshull}
If $A\subset\Z^n$ is a generic $\Lambda$-finite set for some antichain lattice $\Lambda\subset\Z^n$, then $N(A)\cong\hull(A)$.
\end{theorem}

We need a lemma to prove the theorem.

\begin{lemma}\label{lemma137}
If $A\subset\Z^n$ is a generic $\Lambda$-finite set for some antichain lattice $\Lambda\subset\Z^n$, and $F$ is a face of $\hull(A)$, then for every $\alpha\in A$, there is a component $\pi_j(\alpha)$ such that $\pi_j(\alpha)\geq\pi_j(\vee F)$.
\end{lemma}

\begin{proof}$ $

The analogous statement in \cite{MS}, Lemma 6.14 has a finite $A\subset\N^n$, but the hypothesis is never used, and the proof runs identically for infinite $A\subset\Z^n$.
\end{proof}

\begin{proof}{[Theorem \ref{scarfequalshull}]}$ $

Let $F$ be a face of $\hull(A)$ and let $\{\alpha_1, \dots, \alpha_p\}\subset A$ be the points that correspond to the vertices of $F$. That is, $F=\{E_t(\alpha_i)\}$.  Without loss of generality, we can assume that $\pi_i(\bigvee_j\alpha_j)\neq0$.  For a contradiction, assume that $\{\alpha_1, \dots, \alpha_p\}$ is not a face of $N(A)$.  This could occur in two cases:
\begin{enumerate}
 \item There exists $k\in\{1,\dots,p\}$ such that $\bigvee_{j\neq k}\alpha_j=\bigvee_j\alpha_j$. 
\item There exists $\beta\in A$ such that $t^{\beta}\notin F$ and $\beta<\bigvee_j\alpha_j$. (I.e., $\bigvee_j\alpha_j=\beta\vee\bigvee_j\alpha_j$.)
\end{enumerate}

For the first case, if we apply Lemma \ref{lemma137} to $\alpha_k$, then there exists a $j$ such that $\pi_j(\alpha_k)=\pi_j(\bigvee_j\alpha_j)$, and hence there is an element $\alpha_{\ell}$ such that $\pi_j(\alpha_k)=\pi_j(\alpha_{\ell})$.  Since $A$ is generic, there exists $\gamma\in A$ such that $\gamma<<\alpha_k\vee\alpha_{\ell}$, and hence $\gamma\leq\bigvee_j\alpha_j$, contradicting Lemma \ref{lemma137}.

In the second case, if we assume we are distinct from the first case, then for any $\alpha_k\in\{\alpha_1, \dots, \alpha_p\}$, there exists $j$ such that $\pi_j(\alpha_k)=\pi_j(\bigvee_i\alpha_i)\geq\pi_j(\beta)$.  If the inequality is equality, then by genericity, there exists $\beta'<<\bigvee_i\alpha_i$, which is a contradiction to Lemma \ref{lemma137} again, so we have a strict inequality. Having a strict inequality means that $\beta<<\bigvee_i\alpha_i$, again contradicting Lemma \ref{lemma137}.

In both cases, we reached contradictions, and hence every face of $\hull(A)$ is a face of the Scarf complex.  Coupled with Theorem \ref{theorem137}, we have that $\hull(A)\cong N(A)$.  
 
\end{proof}

\begin{corollary}\label{scarfresolves}
 If $A\subset\Z^n$ is a generic $\Lambda$-finite set for some antichain lattice $\Lambda\subset\Z^n$, then $\mathcal{F}_{(N(A))}$ minimally resolves $M_A$ as an $S$-module.
\end{corollary}

\begin{proof}$ $
 
We already have that $\mathcal{F}_{\hull(A)}$ resolves $M_A$, and Theorems \ref{theorem137} and \ref{scarfequalshull} together give us that $\mathcal{F}_{N(A)}$ also resolves it.  The resolution is minimal because no two faces of $N(A)$ have the same degree.
\end{proof}

%\section{Laurent Monomial Modules}

\section{Different Module Structures}\label{6}
Currently, we are operating under the condition that $A\subset\Z^n$ is a generic $\Lambda$-finite set such that $\Lambda\subseteq\Z^n$ is an antichain lattice.  With these assumptions, we have constructed a minimal free resolution of the $S$-submodule $M_A=\{\sum_{\alpha}c_{\alpha}X^{\alpha}\}$ of the Laurent polynomial ring $S^{\pm}=k[x_1^{\pm 1}, \dots, x_n^{\pm 1}]$.  The minimal free resolution we constructed, namely the algebraic Scarf complex of $A$ may only have finitely many nonzero dimensions, but in most dimensions, the module is infinitely generated.  That is, $$(\mathcal{F}_{N(A)})_i=\bigoplus_{\sigma\in N_i(A)}Se_{\sigma}$$ is nonzero for only finitely many $i$, but for the $i$'s for which it is nonzero, there are typically infinitely many $\sigma\in N_i(A)$.  

An underlying structure that we have hitherto underutilized is the grading on $S$, and hence on the $S$-modules.

\subsection{Gradings on S}

The polynomial ring $S=k[x_1, \dots, x_n]$ is graded by $\N^n$, and hence all the $S$-modules we have seen have also been graded by $\N^n$.  Because of this grading, and our ability to associate any monomial in $S$ to a vector in $\N^n$, it will be helpful at times to consider $S$ as the monoid algebra $k[\N^n]$.  This notation will be used when considering gradings that are less common than the $\N^n$-grading. There is a second grading present for many examples that we have yet to consider: the $\Lambda$-grading.

Consider the rings $$S[\Lambda]\cong k[\N^n][\Lambda]=\{\sum_{\alpha,\lambda}c_{\alpha\lambda}X^{\alpha}z^{\lambda}\mid c_{\alpha\lambda}\in k\text{ finitely non-zero }, \alpha\in \N^n, \lambda\in\Lambda\}$$ and $$k[\N^n+\Lambda]=\{\sum_{\beta}c_{\beta}X^{\beta}\mid c_{\beta}\in k\text{ finitely non-zero }, \beta\in\N^n+\Lambda\}$$.

\begin{lemma}\label{lemma1137}Let $A\subset\Z^n$ be a $\Lambda$-finite set such that $\Lambda\subset\Z^n$ is an antichain lattice.  With $M_A=\{\sum_{\alpha}c_{\alpha}t^{\alpha}\mid c_{\alpha}\in k \text{ finitely non-zero }, \alpha\in A+\N^n\}$,
\
\begin{enumerate}
\item $M_A$ is a $k[\N^n+\Lambda]$-module, with action defined by:
$$(x^\beta, t^\alpha)\mapsto t^{\alpha+\beta},\; \alpha\in A+\N^n, \beta\in \N^n+\Lambda,$$ and linearity. 
\item $M_A$ is a $S[\Lambda]$-module, with action defined by:
$$(x^\beta z^{\lambda}, t^\alpha)\mapsto t^{\alpha+\beta+\lambda},\; \alpha\in A, \beta\in \N^n, \lambda\in\Lambda,$$ and linearity. 
\item The set $\{\,t^\alpha\mid\alpha\in A\,\}$ is a minimal set of generators for $M_A$ as an $S$-module.
\item The set $\{\,t^\alpha\mid\alpha\in A_0\,\}$ is a minimal set of generators for $M_A$ as a $k[\N^n+\Lambda]$-module.
\item If $A\subseteq \Lambda+\N^n$ and $A=A+\Lambda+\N^n$, then $M_A$ is an ideal in $k[\N^n+\Lambda]$.
\end{enumerate}
\end{lemma}

\begin{proof}
 \
\begin{enumerate}
\item We have the following equalities that show the result:
\begin{enumerate}
\item $(x^{\beta},t^{\alpha_1}+t^{\alpha_2})\mapsto t^{\alpha_1+\beta}+t^{\alpha_2+\beta}=x^{\beta}t^{\alpha_1}+x^{\beta}t^{\alpha_2}$.
\item $(x^{\beta_1}+x^{\beta_2},t^{\alpha})\mapsto t^{\alpha+\beta_1}+t^{\alpha+\beta_2}=x^{\beta_1}t^{\alpha}+x^{\beta_2}t^{\alpha}$. 
\item $(x^{\beta_1}x^{\beta_2},t^{\alpha})\mapsto t^{\alpha+\beta_1+\beta_2}=x^{\beta_1}t^{\alpha+\beta_2}=x^{\beta_1}(x^{\beta_2}t^{\alpha})$.
\item $(1,t^{\alpha})\mapsto t^{\alpha+0}=t^{\alpha}$.
\end{enumerate}
\item Identical to part 1 with the realization that $\beta+\lambda\in A$, and $\alpha+A\in A$.
\item Let $M_A\ni m=\sum c_{\alpha}t^{\alpha}$ for finitely many $\alpha\in A+\N^n$.  If some $\alpha$ is not in $A$, then there exists an $\eta\in\N^n$ and $\alpha_0\in A$ such that $\alpha=\alpha_0+\eta$.  Then we have that $c_{\alpha}t^{\alpha}=c_{\alpha}t^{\alpha_0+\eta}=c_{\alpha}t^{\eta}t^{\alpha_0}$.  But $t^{\eta}\in S$, so $A$ generates $M_A$ as an $S$-module.
\item Mutatis mutandis with part two, except that now every $\alpha\in A+\N^n$ is written as $\alpha_0+\lambda+\eta$.
\item It suffices to show that $\alpha+\beta\in A+\N^n$ when $\alpha\in A+\N^n$ and $\beta\in\N^n+\Lambda$.  If $\alpha=\lambda_1+\eta_1$, and $\beta=\lambda_2+\eta_2$, then $\alpha+\beta=\lambda_1+\lambda_2+\eta_1+\eta_2$, and since $A=A+\lambda+\N^n$, we have that $\alpha+\beta\in A+\N^n$.
\end{enumerate}

\end{proof}

We have already defined the algebraic Scarf complex to be the Taylor complex supported on $N(A)$.  Implicit in this definition was the consideration of the algebraic Scarf complex as a complex of $S$-modules.  We have now seen that these modules can be considered as $S[\Lambda]$-modules.  

\begin{definition}
 If $A\subset\Z^n$ is a $\Lambda$-finite set such that $\Lambda\subset\Z^n$ is an antichain lattice, then the Taylor complex supported on $N(A)/\Lambda$ considered as a complex of $S[\Lambda]$-modules is $\mathcal{F}^{\Lambda}_{N(A)}$.
\end{definition}

A typical free module in $\mathcal{F}^{\Lambda}_{N(A)}$ would be of the form $$\bigoplus_{\sigma+\Lambda\in N_i(A)/\Lambda}S(e_{\sigma+\Lambda})$$ Due to the onerous nature of this notation, we often will refrain from writing out the modules in detail.

\subsubsection{The Functor $\underline{\white{M}}\otimesS S$}

Let $J$ be the ideal $<1-z^{\lambda}\mid\lambda\in\Lambda>$ in $S[\Lambda]$, and let $\overline{J}$ be the image of $J$ in $k[\N^n+\Lambda]$ under the map $z^{\lambda}x^{\alpha}\mapsto x^{\alpha+\lambda}$.

\begin{lemma}
Let $M$ be an $S[\Lambda]$-module.  Then $S\otimes_{S[\Lambda]}M\cong M/JM$.
\end{lemma}

\begin{proof}$ $ 

Define $b:S\times M\rightarrow M/JM$ by $b(s,m)=sm+JM$.  Then $b$ is surjective and $S$-bilinear.  Furthermore, $b$ is $S[\Lambda]$-bilinear because $b(x^{\lambda}s,m)=\overline{sm}=b(s,x^{\lambda}m)$.  Therefore, $b$ induces an $S$-algebra morphism from $S\otimes_{S[\Lambda]}M$ to $M/JM$, and we can exhibit an inverse.  The kernel of the map $m\mapsto1\otimes m:M\rightarrow S\otimes_{S[\Lambda]}M$ contains $JM$, hence this map induces a morphism $M/JM$ to $S\otimes_{S[\Lambda]}S$.
\end{proof}

Let $A=\Lambda$, under the usual conditions, and consider $M_A\otimesS S=M_{\Lambda}\otimesS S$.  If $I_A=\IL=<X^{\lambda^+}-X^{\lambda^-}\mid\lambda\in\Lambda>$ as usual, then $\IL=\overline{J}\cap S$, and $$M_{\Lambda}\otimesS S\cong k[\N^n+\Lambda]/\overline{J}\cong (S+\overline{J})/\overline{J}\cong S/(\overline{J}\cap S)\cong S/\IL$$

More generally, we can let $M_0$ be the $S$-submodule of $k[\Z^n]$ generated by $\{x^{\alpha}\mid\alpha\in A_0\}$, where $A=A_0+\Lambda$, as usual.  Then notice that if $\alpha\in A$, we can write $\alpha=\alpha_0+\lambda$ for some $\alpha_0\in A_0$ and $\lambda\in\Lambda$, and as such, we have that $x^{\alpha}=x^{\alpha_0}-(1-z^{\lambda})x^{\alpha_0}$.  With this representation of $x^{\alpha}$, we see that $M_A=M_0+JM_A$.  Therefore, we have $$M_A\otimesS S\cong M_A/JM_A\cong (M_0+JM_A)/JM_A\cong M_0/(JM_A\cap M_0)$$

This is too general to say much about, so we will make the assumption that $A\subset\N^n+\Lambda$. With this assumption, we have the following useful lemma.

\begin{lemma}\label{repsinN}
 If $A\subset\N^n+\Lambda$, where $\Lambda\subset\Z^n$ and $\Lambda\cap\N^n=0$, then for any $\alpha\in A$, there are $\alpha_0\in\N^n$ and $\lambda\in\Lambda$ such that $\alpha=\alpha_0+\lambda$.  
\end{lemma}

\begin{proof}$ $

 Let $\alpha\in A$.  Then there exists $\lambda\in\Lambda$ such that $\alpha\in-T_{\lambda}$.  Let $\alpha_0=\alpha-\lambda$.  Then $\alpha_0\in-T_{\alpha-\lambda}\subseteq\N^n$, completing the proof.
\end{proof}

If we choose a generating set for $A$ that is distinguished by being contained in $\N^n$, then using Lemma \ref{lemma1137}, we have that $$M_A\otimesS S\cong M_0(JM_A\cap M_0)= M_0(\IL\cap M_0)\cong (M_0+\IL)/\IL$$

Therefore, in this case, we can identify $M_A\otimesS S$ with the monomial ideal of $S/\IL$ that is generated by $$\{x^{\alpha}+\IL\mid\alpha\in A_0\}$$

Additionally, we have $$k[\Z^n]\otimesS S\cong k[\Z^n]/Jk[\Z^n]\cong k[\Z^n/\Lambda]$$

With this last computation, since $M_A$ is an $S$-submodule of $k[\Z^n]$, we make the claim that $M_A\otimesS S$ is the $S$-submodule of $k[\Z^n/\Lambda]$ generated by the image of $M_0$.  The proof of this claim will come as corollary to Theorem \ref{equivalentcategories}

\subsection{Categorical Equivalence}

Let $\mathcal{A}$ be the category of $S[\Lambda]$-modules with the usual $\Z^n$-grading.  Under the tensor product $\underline{\white{M}}\otimesS S$ that we just worked with, the images are $\Z^n/\Lambda$-graded.  With this setup, let $\mathcal{B}$ be the category of $\Z^n/\Lambda$-graded $S$-modules.

\begin{theorem}\label{equivalentcategories}\emph{[Theorem 9.17, \cite{MS}]}
The tensor product $\pi(\underline{\white{M}})=\underline{\white{M}}\otimesS S:\mathcal{A}\rightarrow\mathcal{B}$ is an equivalence of categories.
\end{theorem}

\begin{corollary}
 If $\mathcal{F}_{\bullet}$ is any $\Z^n$-graded free resolution of $M_A$ over $S[\Lambda]$, then $\pi(\mathcal{F}_{\bullet})$ is a $\Z^n/\Lambda$-graded free resolution of $S/\IL$ over $S$.  Moreover, $\mathcal{F}_{\bullet}$ is minimal if and only if $\pi(\mathcal{F}_{\bullet})$ is minimal.
\end{corollary}

\begin{theorem}\label{thm6.8}
For an antichain lattice $\Lambda\subset\Z^n$, and a $\Lambda$-finite set $A\subset\Z^n$, the following are equivalent:

\begin{enumerate}
 \item The algebraic Scarf complex of $A$, $\mathcal{F}_{N(A)}$.
\item The hull resolution of $A$.

Additionally, they are minimal free $S$-resolutions of $M_A$.
\end{enumerate}
\end{theorem}

\begin{proof}$ $
 
This theorem is a generalization of Theorem 9.24 from \cite{MS}.  The machinery is unchanged, but the setting is broader with the same conclusion and identical proof.
\end{proof}

\begin{corollary}
The isomorphism in Theorem \ref{thm6.8} can be chosen to commute with the $\Lambda$-actions and therefore the isomorphism holds for $S[\Lambda]$-modules and we have a minimal free $S[\Lambda]$-resolution of $M_A$
\end{corollary}

\begin{proof}
This is an identical statement to Corollary \ref{cor5.2}, but with a different application.
\end{proof}

\begin{corollary}
The minimal free resolution of a generic lattice ideal $I_{\Lambda}$ is $\pi(N(\Lambda))$.
\end{corollary}

\section{Application of the Horseshoe Lemma}

To bring everything we have worked on together, we will need the first part of the Horseshoe Lemma.

\begin{lemma}\label{horseshoe}
 Suppose given a commutative diagram $$\begin{array}{cccccccccc} & & & & & & & 0 & & \\ & & & & & & & \downarrow & & \\ \cdots & P_2' & \overset{d'_2}\rightarrow & P_1' & \overset{d'_1}\rightarrow & P_0' & \overset{d'_0}\rightarrow & A' & \rightarrow & 0 \\ & & & & & & & \white{M}\downarrow i_A & & \\ & & & & & & & A & & \\ & & & & & & & \white{M}\downarrow \pi_A & & \\ \cdots & P_2'' & \overset{d''_2}\rightarrow & P_1'' & \overset{d''_1}\rightarrow & P_0'' & \overset{d''_0}\rightarrow & A'' & \rightarrow & 0 \\ & & & & & & &\downarrow & & \\ & & & & & & & 0 & & \end{array}$$

where the column is exact and the rows are projective resolutions.  Set $P_n=P_n'\oplus P_n''$.  Then there exists maps from $P_n$ to $P_{n-1}$ generated from $d'_n$ and $d''_n$ such that $P_{\bullet}$ is a projective resolution of $A$.
\end{lemma}

In our particular case of using cyclic $S$-modules, all of our modules are free and hence projective. Before we arrive at a situation where we can use the Horseshoe Lemma, we need to verify a few conditions first.

\begin{lemma}\label{containsmarkov1}
Let $A\subset\Z^n$ be a generic $\Lambda$-finite set for some antichain lattice $\Lambda\subset\Z^n$ such that $A=A_0+\Lambda$ with $A_0\subset\N^n$ and $A_0\neq\{0\}$. Let $B$ be a minimal Markov basis of $\Lambda$, and assume that $\alpha\nleq\lambda^+$ and $\alpha\nleq\lambda^-$ for all $\lambda\in B$. Then every minimal generating set of $\IL+I_{A_0}\subseteq S$ contains $$\{X^{\lambda^+}-X^{\lambda^-}\mid\lambda\in L\}$$ for some minimal Markov basis $L$ of $\Lambda$.
\end{lemma}

\begin{proof}$ $

Because of Proposition \ref{markovneighbors}, we have that a minimal Markov bases is a subset of a finite set of positive and negative pairs of vectors.  A minimal Markov basis is any subset of this set that chooses one vector from each pair.  As such, the minimal bases only differ by sign patterns, and hence the property $\alpha\in A_0$, $\alpha\nleq\lambda^+$ and $\alpha\nleq\lambda^-$ for all $\lambda\in B$ holds for all Markov bases.  

This condition tells us that $X^{\lambda^+}-X^{\lambda^-}\notin I_{A_0}$.  However, we know that $X^{\lambda^+}-X^{\lambda^-}\in \IL+I_{A_0}$, and hence $X^{\lambda^+}-X^{\lambda^-}\in \IL$.  This holds for all $\lambda\in B$, and hence by the fundamental theorem of Markov bases (Theorem \ref{markdef}), the generating set of $\IL+I_{A_0}$ must contain binomials corresponding to a Markov basis.
\end{proof}

So we have shown that for our generic $\Lambda$-finite sets $A\subseteq\Z^n$ with $\Lambda$-representatives $A_0$, the ideal $\IL+I_{A_0}$ in $S$ can only be written in such a form.

\begin{prop}\label{mainresult}
Let $A\subset\Z^n$ be a generic $\Lambda$-finite set for some antichain lattice $\Lambda\subset\Z^n$ with $\Lambda$-representatives $A_0\subset\N^n$ and $A_0\neq\{0\}$. If $I_{A_0}=<X^{\alpha}\mid\alpha\in A_0>$, then the syzygy modules of the minimal free resolution of $I_{\Lambda}+I_{A_0}$ are submodules of the syzygy modules of $\pi(\mathcal{F}_{N(\Lambda)})\oplus\pi(\mathcal{F}_{N(A)})	$.
\end{prop}

\begin{proof}$ $

Consider the exact sequence $$0\rightarrow \IL \hookrightarrow \IL+I_{A_0} \twoheadrightarrow (I_{A_0}+\IL)/\IL \rightarrow 0$$  By previous arguments, $\pi(\mathcal{F}_{N(\Lambda)})$ and $\pi(\mathcal{F}_{N(A)})$ are free resolutions of $\IL$ and $I_{A_0}$, respectively.  By the Horseshoe Lemma, there exists maps that can be paired with the syzygy modules of $\pi(\mathcal{F}_{N(\Lambda)})\oplus\pi(\mathcal{F}_{N(A)})$ that form a resolution of $\IL+I_{A_0}$. By \cite{P2}, all graded free resolutions contain a minimal graded free resolution, completing the proof.
\end{proof}

Unfortunately, even though $\pi(\mathcal{F}_{N(\Lambda)})$ and $\pi(\mathcal{F}_{N(A)})$ minimally resolve the binomial ideal $\IL\subset S$, and the monomial ideal $(I_{A_0}+\IL)/\IL\subseteq S/\IL$ respectively, the Horseshoe Lemma makes no claim as to the minimality of $\pi(\mathcal{F}_{N(\Lambda)})\oplus\pi(\mathcal{F}_{N(A)})$ as a resolution.  The key to utilizing the Horseshoe Lemma is to understand the maps that are created from the separate resolutions.

\subsection{Lifting Terms}

The proof of the Horseshoe lemma provides a method for defining the new maps of the constructed resolution. In the diagram in Lemma \ref{horseshoe}, the horizontal maps terminating in $A$ are defined first by lifting the map $\epsilon''$ to a map $\overline{\epsilon}'':P_0''\rightarrow A$, and then defining $i_A\circ\epsilon'\oplus\overline{\epsilon}'':P_0'\oplus P_0''\rightarrow A$.  Once this map is constructed, then the process is iterated.  A lifting is defined when we choose a representative of $N_i(A)$ from its $\Lambda$-orbit for each $i$.

\subsection{Lifting Terms in $\Z^3$}

When working with the syzygy modules of the ideal $I_A=\IL+I_{A_0}$, we have several symbols that must be handled very carefully.  In particular, if we have chosen a set of representatives for each $\Lambda$-orbit of $N(A)$, then each face $F$ has a representative face $F'$ such that  $F=F'+\lambda$ for some $\lambda\in\Lambda$.  Additionally, each face of $F$ has its own representative that may or may not be a face of $F$.  These considerations lead us to the following potential problem.  In $N(A)/\Lambda$, we have generators of our modules of the from $e_{\sigma+\Lambda}=e_{\overline{\sigma}}$; in $N(A)$, it would appear that we have generators of the form $e_{\sigma}$, but that is only true of the representative we chose for the lifting.  As such, we need a definition for $e_{\sigma}$ if $\sigma$ is not a representative.

\begin{lemma}\label{abcd}
Let $\Lambda\subset\Z^3$ be an antichain lattice with minimal Markov basis $\{\lambda_i\}$, and let $g\in\Lambda$.  Then there exists $\{c_i\}\subset S$ such that $$X^{g^+}-X^{g^-}=\sum_ic_i(X^{\lambda_i^+}-X^{\lambda_i^-})$$
\end{lemma}

\begin{proof} $ $
 
By definition of $\IL$, if $g\in\Lambda$, then $X^{g^+}-X^{g^-}\in\IL$, and by the fundamental theorem of Markov bases, $\{X^{\lambda_i^+}-X^{\lambda_i^-}\}$ generates $\IL$.
\end{proof}

\begin{definition}\label{1234}
Let $\Lambda$ be an antichain lattice in $\Z^3$ with minimal Markov basis $\{\lambda_i\}$, and let $P_0=\pi(\mathcal{F}_{N(A)})_0\oplus\pi(\mathcal{F}_{N(\Lambda)})_0=(\displaystyle{\bigoplus_{\sigma\in A_0}}Se_{\sigma})\oplus Se_{\lambda_1}\oplus Se_{\lambda_2}\oplus Se_{\lambda_3}$, where $A_0$ is a set of $\Lambda$-representatives of $N_0(A)$.  Let $g\in\Lambda$ such that $X^{g^+}-X^{g^-}=\sum_ic_i(X^{\lambda_i^+}-X^{\lambda_i^-})$ and let $C=\{c_i\}$. Then we define $e_g(C)=\sum c_ie_{\lambda_i}$.
\end{definition}

\begin{remark}\label{uptosomething}
For all $C\subset S$ satisfying Definition \ref{1234}, if $d_0:P_0\rightarrow I_A$, then $d_0(e_g(C))=X^{g^+}-X^{g^-}$.  Because of this, we can relax the notational dependence of $e_g$ on $C$.
\end{remark}

We now have a way to consider the symbol $e_g$ in terms of the symbols $e_{\lambda_i}$, which are generators of the $0^{th}$ dimensional module in the resolution of $\IL$.  These symbols will often arise in symbolic computations, and needed to be addressed before we proceeded.

\begin{lemma}\label{N_1lifting}
Let $A\subset\Z^3$ be a generic $\Lambda$-finite set for a codimension 1 lattice $\Lambda\subset\Z^3$.  Let $B,C\in\N^3$, $f,g\in\Lambda$ and $\sigma=\{B+f,C+g\}\subset N_1(A)$ oriented from $B+f$ to $C+g$. If $P_0=(\displaystyle{\bigoplus_{\sigma\in A_0}}Se_{\sigma})\oplus Se_{\lambda_1}\oplus Se_{\lambda_2}\oplus Se_{\lambda_3}$, $I_A=\IL+I_{A_0}$, and $d_0:P_0\rightarrow I_A$, then $$d_0(e_\sigma)=X^{S-(C+g)}e_C-X^{S-(B+f)}e_B+X^{S-g^+}e_g-X^{S-f^+}e_f$$ where $S=(B+f)\vee(C+g)$.
\end{lemma}

\begin{proof}$ $

The first two terms of the expression are $d'(e_{\sigma})$ when we consider $\sigma$ as an element of $N_1(A)/\Lambda$.  So we need to show that if we attempted to use this same map for $d(e_{\sigma})$, then we would have the second pair of terms of the expression left over.  Computing, $$dd(e_{\sigma})=d(X^{S-(C+g)}e_C-X^{S-(B+f)}e_B)=X^{S-(C+g)}d(e_C)-X^{S-(B+f)}d(e_B)$$ $$=X^{S-g}-X^{S-f}\neq 0$$

Therefore, we need to add an expression to $X^{S-(C+g)}e_C-X^{S-(B+f)}e_B$ such that applying $d$ to that expression will give us $X^{S-g}-X^{S-f}$.  That expression is exactly $X^{S-g^+}e_g-X^{S-f^+}e_f$.  Applying $d$, we get $$X^{S-g^+}(X^{g^+}-X^{g^-})-X^{S-f^+}(X^{f^+}-X^{f^-})$$ $$=X^S-X^{S-g^++g^-}-X^S+X^{S-f^++f^-}=X^{S-g}-X^{S-f}$$ as required.
\end{proof}

\begin{remark}
In Lemma \ref{N_1lifting}, even though we were equipped with Definition \ref{1234}, it appears as though we did not use it.  This is because if we had replaced $e_g$ with $\sum c_ie_{\lambda_i}$, all the terms would have canceled just as if we had left $e_g$ in the computation.  This situation repeats itself often in similar computations, and when we are able, we will use the analogues of $e_g$ directly in future computations with the understanding that they are only symbolic.
\end{remark}

Since we are in $\Z^3$, need only have Definition \ref{1234} and a similar definition for faces to handle all possible cases we might run into.

\begin{lemma}\label{efgh}
Let $A\subset\Z^3$ be a generic $\Lambda$-finite set for some codimension 1 antichain lattice $\Lambda\subset\Z^3$ with minimal Markov basis $\{\lambda_i\}$. Let $A_1$ be a set of $\Lambda$-representatives of $N_1(A)$. Suppose $t\in N_1(A)$ with endpoints $B+f$ and $C+g$.  Let $t^r\in N_1(A)$ be the representative of $t$ and assume that $t=t^r+h$ with $h\in\Lambda$. Let $c_i, c_i', d_i, d_i'$ be the coefficients described in Lemma \ref{abcd} for $g-h, g, f-h$, and $f$, respectively. If $P_1=(\displaystyle{\bigoplus_{\sigma\in A_1}}Se_{\sigma})\oplus Se_{p_1}\oplus Se_{p_2}$ where $p_1,p_2$ are as in Lemma \ref{lambdares}, and $P_0=(\displaystyle{\bigoplus_{\sigma\in A_0}}Se_{\sigma})\oplus Se_{\lambda_1}\oplus Se_{\lambda_2}\oplus Se_{\lambda_3}$ and $d_1:P_1\rightarrow P_0$, then, symbolically, $$d_1(e_t)-d_1(e_{t^r})=\sum(c_iX^{\vee t^r-(g-h)^+}-c_i'X^{\vee t-g^+}-d_iX^{\vee t^r-(f-h)^+}+d_i'X^{\vee t-f^+})e_{\lambda_i}$$
\end{lemma}

\begin{proof}$ $

We compute with the understanding that $d(e_t)$ is a symbolic computation. To aid in the computation, we can create a diagram out of the hypothesis as follows:

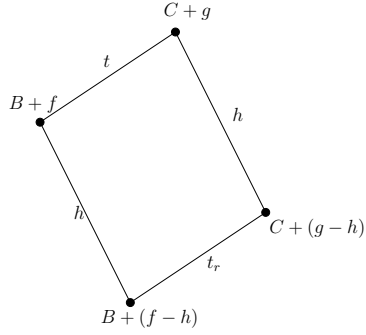
\begin{center}
\scalemath{.6}{\begin{tikzpicture}[line cap=round,line join=round,>=triangle 45,x=1.0cm,y=1.0cm]
%\clip(-4.980000000000003,-2.700000000000001) rectangle (22.700000000000014,11.580000000000004);
\draw (2.0,6.0)-- (4.0,2.0);
\draw (4.0,2.0)-- (7.0,4.0);
\draw (7.0,4.0)-- (5.0,8.0);
\draw (5.0,8.0)-- (2.0,6.0);
\draw (1.1600000000000006,6.720000000000002) node[anchor=north west] {$B+f$};
\draw (4.600000000000002,8.780000000000003) node[anchor=north west] {$C+g$};
\draw (3.200000000000002,2.000000000000001) node[anchor=north west] {$B+(f-h)$};
\draw (6.940000000000004,4.000000000000002) node[anchor=north west] {$C+(g-h)$};
\draw (2.6000000000000014,4.280000000000001) node[anchor=north west] {$h$};
\draw (6.120000000000004,6.460000000000003) node[anchor=north west] {$h$};
\draw (3.2600000000000016,7.620000000000003) node[anchor=north west] {$t$};
\draw (5.560000000000003,3.160000000000001) node[anchor=north west] {$t_r$};
\begin{scriptsize}
\draw [fill=black] (2.0,6.0) circle (2.5pt);
\draw [fill=black] (5.0,8.0) circle (2.5pt);
\draw [fill=black] (4.0,2.0) circle (2.5pt);
\draw [fill=black] (7.0,4.0) circle (2.5pt);
\end{scriptsize}
\end{tikzpicture}}
\end{center}

$$d(e_{t^r})=X^{\vee t^r-(C+g-h)}e_C-X^{\vee t^r-(B+f-h)}e_B + X^{\vee t^r-(g-h)^+}\sum c_ie_{\lambda_i}-X^{\vee t^r-(f-h)^+}\sum d_ie_{\lambda_i}$$

$$d(e_t)=X^{\vee t-(C+g)}e_C-X^{\vee t-(B+f)}e_B + X^{\vee t-g^+}\sum c'_ie_{\lambda_i}-X^{\vee t-f^+}\sum d'_ie_{\lambda_i}$$

Taking the difference and rearranging, we get 

\begin{equation}\label{eq137}(X^{\vee t^r - (C+g-h)}-X^{\vee t -(C+g)})e_C-(X^{\vee t^r - (B+f-h)}-X^{\vee t -(B+f)})e_B$$
$$+\sum(c_iX^{\vee t^r-(g-h)^+}-c'_iX^{\vee t-g^+})e_{\lambda_i}-\sum(d_iX^{\vee t^r-(f-h)^+}-d'_iX^{\vee t-f^+})e_{\lambda_i}\end{equation}

Notice now that $$\vee t^r-(C+g-h)=(B+f-h)\vee(C+g-h)-(C+g-h)=(B+f)\vee(C+g)-h-(C+g-h)$$
$$=(B+f)\vee(C+g)-(C+g)=\vee t-(C+g),$$ so the first parenthetical expression of \ref{eq137} is 0, and by an identical computation, the second parenthetical expression is also 0.  This leaves us with the desired result
\end{proof}

Definition \ref{exists?} will exemplify the nature of Remark \ref{uptosomething} in the sense that we will define the term exactly by how it acts under the mapping, and not how it acts as a module element.  As in Definition \ref{N_1lifting}, we will not need to reference the defining set in practice, and will supress the notation.

\begin{definition}\label{exists?}
Under the conditions of Lemma \ref{efgh}, we define $d(e_t(\mathcal{B}))=d(e_{t^r})+\sum b_id(e_{p_i})$, where $p_i$ is as in Lemma \ref{lambdares}, $\mathcal{B}=\{b_i\}$, and the $b_i$ satisfy $$\sum(c_iX^{\vee t^r-(g-h)^+}-c_i'X^{\vee t-g^+}-d_iX^{\vee t^r-(f-h)^+}+d_i'X^{\vee t-f^+})e_{\lambda_i}=\sum b_id(e_{p_i})$$
\end{definition}

We will call the expressions computed for Definitions \ref{1234} and \ref{exists?} lifting terms in their respective dimensions.

\subsection{Example}
To conclude, we will compute a example using the tools developed here.

\begin{example}

Let $\Lambda$ be the lattice generated by $\{(-1,2,-1),(3,-1,-1)\}$in $\Z^3$, and let $A_0=\{\alpha\}=\{(1,2,0)\}$.  A minimal Markov basis of $\Lambda$ is $\{\lambda_1,\lambda_2,\lambda_3\}=\{(-1,2,-1),(3,-1,-1),(-2,-1,2)\}$\footnote{Markov basis computations can be performed in 4ti2, \cite{ti}}.  For representatives, we will choose \linebreak $A_1=\{r,s,t\}= \{\{(1,2,0),(4,3,-1)\},\{(1,2,0),(3,3,-2)\},\{(1,2,0),(0,4,-1)\}\}$, and $A_2=\{u,v\}=\{\{(1,2,0),(0,4,-1),(3,3,-2)\},\{(1,2,0),(3,3,-2),(4,3,-1)\}\}$ with the orientations as listed, and we obtain the following diagram for $N(A)/\Lambda$ where the representatives are indicated by solid lines or filled in circles, and the suprema labeled in the appropriate places.

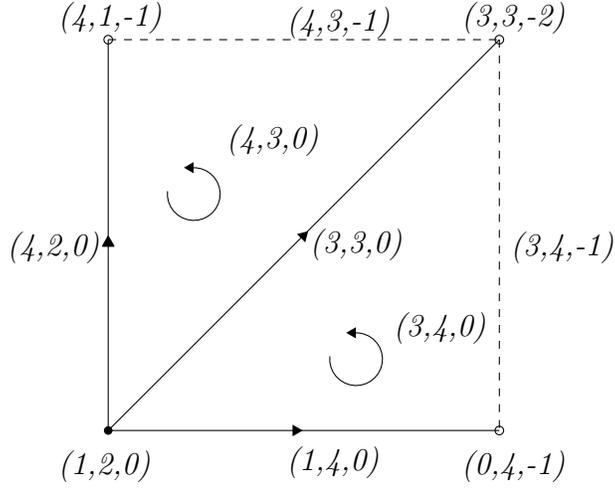
\begin{center}
 \begin{tikzpicture}[scale=.65][line cap=round,line join=round,>=triangle 45,x=1.0cm,y=1.0cm]
%\clip(-11.6,-3.98) rectangle (19.96,14.37);
\fill[fill=black,fill opacity=1.0] (-0.12,3.77) -- (0.12,3.77) -- (0,3.98) -- cycle;
\fill[fill=black,fill opacity=1.0] (3.77,0.12) -- (3.77,-0.12) -- (3.98,0) -- cycle;
\fill[fill=black,fill opacity=1.0] (3.86,4.02) -- (4.03,3.85) -- (4.09,4.08) -- cycle;
\draw (0,0)-- (0,8);
\draw [dash pattern=on 3pt off 3pt] (0,8)-- (8,8);
\draw [dash pattern=on 3pt off 3pt] (8,8)-- (8,0);
\draw (8,0)-- (0,0);
\draw (0,0)-- (8,8);
\draw (-0.12,3.77)-- (0.12,3.77);
\draw (0.12,3.77)-- (0,3.98);
\draw (0,3.98)-- (-0.12,3.77);
\draw [shift={(5.07,1.46)}] plot[domain=-3.25:1.69,variable=\t]({1*0.54*cos(\t r)+0*0.54*sin(\t r)},{0*0.54*cos(\t r)+1*0.54*sin(\t r)});
\draw [shift={(1.75,4.84)}] plot[domain=-3.25:1.69,variable=\t]({1*0.54*cos(\t r)+0*0.54*sin(\t r)},{0*0.54*cos(\t r)+1*0.54*sin(\t r)});
\draw (-1.25,-0.17) node[anchor=north west] {(1,2,0)};
\draw (7,-0.17) node[anchor=north west] {(0,4,-1)};
\draw (-1.25,9) node[anchor=north west] {(4,1,-1)};
\draw (7,9) node[anchor=north west] {(3,3,-2)};
\draw (-2.3,4.3) node[anchor=north west] {(4,2,0)};
\draw (3.42,-0.1) node[anchor=north west] {(1,4,0)};
\draw (8,4.3) node[anchor=north west] {(3,4,-1)};
\draw (3.42,8.9) node[anchor=north west] {(4,3,-1)};
\draw (3.9,4.37) node[anchor=north west] {(3,3,0)};
\draw (2.19,6.5) node[anchor=north west] {(4,3,0)};
\draw (5.61,2.66) node[anchor=north west] {(3,4,0)};
\begin{scriptsize}
\fill [color=black] (0,0) circle (2.5pt);
\draw [color=black] (0,8) circle (2.5pt);
\draw [color=black] (8,0) circle (2.5pt);
\draw [color=black] (8,8) circle (2.5pt);
\fill [color=black,shift={(5,2)},rotate=90] (0,0) ++(0 pt,3.75pt) -- ++(3.25pt,-5.625pt)--++(-6.5pt,0 pt) -- ++(3.25pt,5.625pt);
\fill [color=black,shift={(1.68,5.38)},rotate=90] (0,0) ++(0 pt,3.75pt) -- ++(3.25pt,-5.625pt)--++(-6.5pt,0 pt) -- ++(3.25pt,5.625pt);
\end{scriptsize}
\end{tikzpicture}
\end{center}

We must first compute the resolution of $(\IL+I_{A_0})/\IL$ using the coefficients computed from Definition \ref{cellrescoeffs}.  For example, the relation associated to the edge $t$ is $xze_{\overline{\alpha}}-y^2e_{\overline{\alpha}}=(xz-y^2)e_{\overline{\alpha}}$.\footnote{We are making a slight abuse of the diagram here: the diagram should only explicity be used for the resolution of $I_A$, but if we ignore the repeated edges, we can make use of it as a guide for the resolution of $(I_A+\IL)/\IL.$} The relation associated to the face $u$ is $x^2e_{\overline{t}}+ze_{\overline{r}}-ye_{\overline{s}}$.  Omitting the details of the remaining computations, we have that the resolution of $(\IL+I_{A_0})/\IL$, $\pi(\mathcal{F}_{N(A)})$ is 

$$\begin{array}{ccccccc}Se_{\overline{u}}\oplus Se_{\overline{v}} & \rightarrow & Se_{\overline{r}}\oplus Se_{\overline{s}}\oplus Se_{\overline{t}} & \rightarrow & Se_{\overline{\alpha}} & \rightarrow & (\IL+I_{A_0})/\IL\\
   e_{\overline{u}} & \mapsto & x^2e_{\overline{t}}+ze_{\overline{r}}-ye_{\overline{s}} & & & & \\
   e_{\overline{v}} & \mapsto & ze_{\overline{t}}+ye_{\overline{r}}-xe_{\overline{s}} & & & & \\
   & & e_{\overline{t}} & \mapsto & (xz-y^2)e_{\overline{\alpha}} & & \\
   & & e_{\overline{r}} & \mapsto & (y-x^3z)e_{\overline{\alpha}} & & \\
   & & e_{\overline{s}} & \mapsto & (z^2-x^2y)e_{\overline{\alpha}} & & \\
   & & & & e_{\overline{\alpha}} & \mapsto & xy^2+\IL
  \end{array}$$

Using the same diagram for the lifting computations, we will again show one example from each dimension.  The edge $t$ is of the form $\{\alpha,\alpha+\lambda_1\}$ oriented from $\alpha$ to $\alpha+\lambda_1$.  Making the substitutions into Lemma \ref{N_1lifting}, we have that $B=\alpha$, $f=0$ (consequently, $e_f=0$), $C=\alpha$, and $g=\lambda_1$.  Therefore, our lifted map will be $$d_1(e_t)=X^{S-(\alpha+\lambda_1)}e_{\alpha}-X^{S-\alpha}e_{\alpha}+X^{S-\lambda_1^+}e_{\lambda_1}$$ $$=xze_{\alpha}-y^2e_{\alpha}+xy^2e_{\lambda}$$$$=(xz-y^2)e_{\alpha}+xy^2e_{\lambda_1}$$

We will show the use of Lemma \ref{efgh} for the face $u$.  Notice that the edges $t$ and $s$ are already representatives, so we will only need a lifting term for our tranlsation of the edge $r$.  From Lemma \ref{efgh}, we have that $f=\lambda_1$, $g=-\lambda_3$, and $h=\lambda_1$.  Additionally, we already have computed that $\vee r=(3,4,-1)$ and $\vee r^r=(4,2,0)$.  What is left to compute are the $c_is, c'_is, d_is,$ and $d'_is$.  The three easy cases are $c'_i,d_i,$ and $d_i'$: $f=\lambda_1$ implies $d'_1=1$ and $d'_2=d'_3=0$; $g=-\lambda_3$ implies $c'_3=-1$ and $c'_1=c'_2=0$; and $f=h$ implies $d_i=0$ for all $i$.  For $g-h$, we need to write $X^{(g-h)^+}-X^{(g-h)^-}=\sum c_i(X^{\lambda_i^+}-X^{\lambda_i^-})$.  Since $g-h=\lambda_2$, we have that $c_2=1$ and $c_1=c_3=0$.  

Continuing, we have $$d_1(e_r^r)-d_1(e_r)=X^{(3,4,-1)-(0,2,0)}e_{\lambda_1} + X^{(4,2,0)-(3,0,0)}e_{\lambda_2}-X^{(3,4,-1)-(2,1,0)}e_{\lambda_3}$$ $$=x^3y^2z^{-1}e_{\lambda_1}+xy^2e_{\lambda_2}+xy^3z^{-1}e_{\lambda_3}$$

To use this, $$d_2(e_u)=x^2e_t+ze_r-ye_s$$
$$=x^2e_{t^r}+z(e_{r^r}-xy^2(x^2z^{-1}d^{-1}_1(e_{\lambda_1})+d^{-1}_1(e_{\lambda_2})-yz^{-1}d^{-1}_1(e_{\lambda_3})))-ye_{s^r}$$
$$=x^2e_{t^r}+ze_{r^r}-ye_{s^r}-xy^2e_{p_1}$$

Omitting the remaining similar computations, we have 

$$\scalemath{.8}{\begin{array}{ccccccc} Se_u\oplus Se_v & \overset{d_2}\rightarrow & Se_{p_1}\oplus Se_{p_2} \oplus Se_r \oplus Se_s\oplus Se_t & \overset{d_1}\rightarrow & Se_{\lambda_1}\oplus Se_{\lambda_2}\oplus Se_{\lambda_3} \oplus Se_{\alpha} & \overset{d_0}\rightarrow & I_A \\
   e_u & \mapsto & x^2e_{t^r}+ze_{r^r}-ye_{s^r}-xy^2e_{p_1} & & & & \\
e_v & \mapsto & ze_{t^r}+ye_{r^r}-xe_{s^r}-xy^2e_{p_2} & & & & \\
& & e_{p_1} & \mapsto & x^2e_{\lambda_1}+ze_{\lambda_2}-ye_{\lambda_3} & & \\
& & e_{p_2} & \mapsto & ze_{\lambda_1}+ye_{\lambda_2}-xe_{\lambda_3} & & \\
& & e_r & \mapsto & xy^2e_{\lambda_2} - (x^3-yz)e_{\alpha} & & \\
& & e_s & \mapsto & xy^2e_{\lambda_3} - (x^2y-z^2)e_{\alpha} & & \\
& & e_t & \mapsto & xy^2e_{\lambda_1} - (y^2-xz)e_{\alpha} & & \\
& & & & e_{\lambda_1} & \mapsto y^2-xz \\
& & & & e_{\lambda_2} & \mapsto x^3-yz \\
& & & & e_{\lambda_3} & \mapsto x^2y-z^2 \\
& & & & e_{\alpha} & \mapsto xy^2 
  \end{array}}$$
\end{example}

\begin{remark}
During long computations, such as we have just completed, many small perturbations occur without mention, such as rearranging terms, or moving negative signs around.  One notable point from the previous computation was the occurence of $z^{-1}$ during an intermediate step.  Although $z^{-1}\notin S$, the end result justified the means, so we choose to ignore the phenomenon.
\end{remark}

\section{Conclusion}

The main result of this paper is Proposition \ref{mainresult} together with Lemmas \ref{N_1lifting} and \ref{efgh}.  The general case runs identically, where we perform formal computations and match it with what our representatives should look like, defining the lifting terms in higher dimensions accordingly.  The result is analogous, but messier, versions of Lemmas \ref{N_1lifting} and \ref{efgh} for any dimension.

The author has recently become acquainted with the work of L\"{u} in \cite{lu1} and \cite{lu2} in which one can make very nice statements concerning the minimality of resolutions obtained from applications of the horseshoe lemma.  In the three dimensional case covered here, minimality was essentially free, but in higher dimensions, the computation of all the lifting terms is a daunting undertaking.  Using these new results has the potential to prove some very clean statements about minimality, and this will be explored in the future.

An additional line of research lies in studying ideals of the form $I=<X^{\lambda_i^+}-X^{\lambda_i^-} | \lambda_i \text{ generates } \Lambda>$ for some generic antichain lattice $\Lambda$.  This is different from the existing case in that we are not requiring a full Markov basis, just a full lattice basis.  The idea that is supported by preliminary computations is that one can pass from the deficient ideal to the full lattice ideal $I_{\Lambda}$, then perform the algorithm outlined in this paper, then pass from that resolution into another resolution via a simple algorithm.  This has been shown to work in three dimensions, and further cases will be studied.

\end{document}